\documentclass[a4paper,british]{article}
\usepackage[T1]{fontenc}
\usepackage[latin9]{inputenc}
\usepackage{amsmath}
\usepackage{amsthm}
\usepackage{amssymb}
\usepackage{graphicx}
\usepackage[numbers]{natbib}

\makeatletter

\numberwithin{equation}{section}
\theoremstyle{plain}
\newtheorem{thm}{\protect\theoremname}
  \theoremstyle{plain}
  \newtheorem{assumption}[thm]{\protect\assumptionname}
  \theoremstyle{definition}
  \newtheorem{example}[thm]{\protect\examplename}
  \theoremstyle{plain}
  \newtheorem{prop}[thm]{\protect\propositionname}
  \theoremstyle{remark}
  \newtheorem{rem}[thm]{\protect\remarkname}
  \theoremstyle{plain}
  \newtheorem{lem}[thm]{\protect\lemmaname}

\usepackage{a4wide}
\usepackage{bbm}

\renewcommand{\hat}{\widehat}
\renewcommand{\tilde}{\widetilde}

\makeatother

\usepackage{babel}
  \providecommand{\assumptionname}{Assumption}
  \providecommand{\examplename}{Examples}
  \providecommand{\lemmaname}{Lemma}
  \providecommand{\propositionname}{Proposition}
  \providecommand{\remarkname}{Remark}
\providecommand{\theoremname}{Theorem}

\begin{document}
\global\long\def\phi{\varphi}
\global\long\def\epsilon{\varepsilon}
\global\long\def\theta{\vartheta}
\global\long\def\E{\mathbb{E}}
\global\long\def\Var{\operatorname{Var}}
\global\long\def\N{\mathbb{N}}
\global\long\def\Z{\mathbb{Z}}
\global\long\def\R{\mathbb{R}}
\global\long\def\F{\mathcal{F}}
\global\long\def\le{\leqslant}
\global\long\def\ge{\geqslant}
\global\long\def\MT{\ensuremath{\clubsuit}}
\global\long\def\1{\mathbbm1}
\global\long\def\d{\mathrm{d}}
\global\long\def\P{\mathbb{P}}
\global\long\def\Id{\operatorname{Id}}
\global\long\def\Span{\operatorname{span}}
 \global\long\def\subset{\subseteq}
\global\long\def\supset{\supseteq}
\global\long\def\argmin{\arg\,\min}
\global\long\def\bull{{\scriptstyle \bullet}}
\global\long\def\supp{\operatorname{supp}}
\global\long\def\sgn{\operatorname{sign}}

\title{Bayesian inverse problems with unknown operators}

\author{Mathias Trabs}

\date{Universität Hamburg}
\maketitle
\begin{abstract}
We consider the Bayesian approach to linear inverse problems when
the underlying operator depends on an unknown parameter. Allowing
for finite dimensional as well as infinite dimensional parameters,
the theory covers several models with different levels of uncertainty
in the operator. Using product priors, we prove contraction rates
for the posterior distribution which coincide with the optimal convergence
rates up to logarithmic factors. In order to adapt to the unknown
smoothness, an empirical Bayes procedure is constructed based on Lepski's
method. The procedure is illustrated in numerical examples.
\end{abstract}


\section{Introduction}

Bayesian procedures to solve inverse problems became increasingly
popular in the last years, cf. \citet{stuart2010}. In the inverse
problem literature the underlying operator of the forward problem
is typically assumed to be known. In practice, there might however
be some uncertainty in the operator 
which has to be taken into account by the procedure. While there are some frequentist approaches in the statistical
literature to solve inverse problems with an unknown operator, the
Bayesian point of view has not yet been analysed. The aim of this
work is to fill this gap.

Let $f\in L^{2}(\mathcal{D})$ be a function on a domain $\mathcal{D}\subset\R^{d}$
and $K_{\theta}\colon L^{2}(\mathcal{D})\to L^{2}(\mathcal{Q})$,
$\mathcal{Q}\subset\R^{q}$, be an injective, continuous linear operator
depending on some parameter $\theta\in\Theta$. We consider the linear
inverse problem
\begin{equation}
Y=K_{\theta}f+\epsilon Z,\label{eq:obsY}
\end{equation}
where $Z$ Gaussian white noise in $L^{2}(\mathcal{Q})$ and $\epsilon>0$
is the noise level which converges to zero asymptotically. If the
operator $K_{\theta}$ is known, the inverse problem to recover $f$
non-parametrically, i.e. as element of the infinite dimensional space
$L^{2}(\mathcal{D})$, from the observation $Y$ is well studied,
see for instance \citet{cavalier2008}. The Bayesian approach has
been analysed by \citet{knapikEtAl2011} with Gaussian priors, by
\citet{ray2013} non-conjugate priors and many subsequent articles
including \citep{agapiouEtAl2013,Agapiou2014,knapikEtAl2016,knapikEtAl2013}.
Also non-linear inverse problems have been successfully solved via
Bayesian methods, for example, \citep{bochkina2013,dashtiEtAl2013,nickl2017,nicklSoehl2017Bernstein,nicklSoehl2017,vollmer2013}.

Focussing on linear inverse problems, we will extend the Bayesian
methodology to unknown operators. To this end, the unknown parameter
$\theta\in\Theta$ is introduced in (\ref{eq:obsY}) where $K_{\theta}$
may depend non-linearly on $\theta$. Unknown operators are relevant in numerous applications. Examples include semi-blind and blind deconvolution for image analysis. Therein, the operator is given by $K_\theta f=\int g_\theta(\cdot-y)f(y)dy$ with some unknown convolution kernel $g_\theta$ \cite{BurgerScherzer2001,JustenRamlau2006,StueckEtAl2012}. More general integral operators such as singular layer potential operators appear in the context of partial differential equations, cf. examples in \cite{CohenEtAl2004,gugushviliEtAl2018}. If the coefficients of the underlying PDE are unknown, then the operator itself is only partially known. A typical example of this type is the backwards heat equation where the solution $u$ of the PDE $\frac\partial{\partial t}u=\theta\Delta u$ (with Dirichlet boundary conditions) is observed at some time $t$ and the aim is to estimate the initial value function $f=u(0,\cdot)$. Here, we take into account an unknown diffusivity parameter $\theta>0$. The solution $u(t,\cdot)$ depends linearly on $f$ and the resulting operator admits a singular value decomposition (SVD) with respect to the sine basis and with $\theta$ dependent singular values $\rho_{\theta,k}=e^{-\theta \pi^2k^2t},k\ge1$, cf. Section~\ref{sec:Applications}. In particular, the resulting inverse problem is severely ill-posed.

Even without measurement errors the target function $f$ is in general not identifiable any more for unknown operators, i.e., there may be several solutions $(\theta,f)$ to the equation $Y=K_\theta f$. For instance, if $K_\theta$ admits a SVD $K_\theta\phi_k=\rho_{k}\psi_k$ for a orthonormal systems $(\phi_k)_{k\ge1},(\psi_k)_{k\ge1}$ and unknown singular values $\theta=(\rho_k)_{k\ge1}$, then we have $K_{\theta}f=K_{\theta/a}(af)$ for any function $f\in L^2(\mathcal D)$ and any scalar $a>0$. We thus require some extra information.

There are different approaches in the inverse problem literature to deal with this identifiability problem, particularly in the context of semi-blind or blind deconvolution. One approach is to find the so called minimum norm solution which has a minimal distance to some a priori estimates for $\theta$ and $f$, cf. \cite{BurgerScherzer2001,JustenRamlau2006}. Another idea is to assume that some approximation of the unknown operator is available for the reconstruction of $f$, cf. \cite{JohannesEtAl2011, StueckEtAl2012}. Similarly, we may assume to have some noisy observation of the unknown parameter $\theta$ which then allows to construct an estimator for $K_\theta$. 

In this paper we will study this last setting. More precisely, we suppose that the parameter set $\Theta$ is (a subset of) some Hilbert space and we consider the additional sample
\begin{equation}
T=\theta+\delta W\label{eq:obsT}
\end{equation}
where $W$ is white noise on $\Theta$, independent of $Z$, and $\delta>0$
is some noise level. Thereby, $\theta$ is considered as a nuisance
parameter and we will not impose any regularity assumptions on $\theta$.
Our aim is the estimation of $f$ from the observations (\ref{eq:obsY})
and (\ref{eq:obsT}). 
This setting includes several exemplary models with different levels
of uncertainty in the operator $K_{\theta}$:
\begin{enumerate}
\item[A] If $\Theta\subset\R^{p}$, we have a parametric characterization
of the operator $K_{\theta}$ and $T$ can be understood as an independent
estimator for $\theta$. 
\item[B] \citet{cavalierHengartner2005} have studied the case where the eigenfunctions
of $K_{\theta}$ are known, but only noisy observations of the singular
values $(\rho_{k})_{k\ge1}$ are observed: $T_{k}=\rho_{k}+\delta W_{k},k\ge1$,
with i.i.d. standard normal $(W_{k})_{k}$. In this case $\Theta=\ell^{2}$,
supposing $K_{\theta}$ is Hilbert-Schmidt, and $\theta=(\rho_{k})_{k}$
is the sequences of singular values of $K_{\theta}$. 
\item[C] \citet{efromovichKoltchinskii2001}, \citet{HoffmannReiss2008} as well as \citet{marteau2006}
have assumed the operator as completely unknown and considered additional
observations of the form
\[
L=K+\delta W
\]
where the operator $L$ is blurred by some independent white noise
$W$ on the space of linear operators from $L^{2}(\mathcal{D})$ to
$L^{2}(\mathcal{Q})$ with some noise level $\delta$. Fixing basis
$(e_{k})$ and $(h_{k})$ of $L^{2}(\mathcal{D})$ and $L^{2}(\mathcal{Q})$,
respectively, $K$ is characterised by the infinite matrix $\theta=(\langle Ke_{k},h_{l}\rangle)_{k,l\ge1}\in\R^{\N^{2}}$
and $W$ can be identified with the random matrix $(\langle We_{k},g_{l}\rangle)_{k,l\ge1}$
consisting of i.i.d. standard Gaussian entries.
\end{enumerate}
In contrast to the just mentioned articles \citep{cavalierHengartner2005,efromovichKoltchinskii2001,HoffmannReiss2008,marteau2006}, we will investigate the Bayesian approach. We thus put a prior distribution $\Pi$ on $(f,\theta)\in L^{2}(\mathcal{D})\times\Theta$. Denoting the probability density of $(Y,T)$ under the parameters $(f,\theta)$ with respect to some reference measure by $p_{f,\theta}$, the posterior distribution given the observations $(Y,T)$ is given by Bayes' theorem:
\begin{equation}\label{eq:posteriorIntro}
  \Pi(B|Y,T)=\frac{\int_Bp_{f,\theta}(Y,T)\d\Pi(f,\theta)}{\int_{L^2\times\Theta}p_{f,\theta}(Y,T)\d\Pi(f,\theta)}
\end{equation}
for measurable subsets $B\subset L^2(\mathcal{D})\times\Theta$. Due to the white noise model, the density $p_{f,\theta}$ inherits the nice structure from the normal distribution, cf. Section~\ref{sec:posterior}. Although we cannot hope for nice conjugate pairs of prior and posterior distribution due to the non-linear structure of $(f,\theta)\mapsto K_\theta f$, there are efficient Markov chain Monte Carlo algorithms to draw from $\Pi(\cdot|Y,T)$, cf. \citet{tierney1994}.

To analyse the behaviour of the posterior distribution, we will take a frequentist point of view and assume the observations
are generated under some true, but unknown $f_{0}\in L^{2}(\mathcal{D})$
and $\theta_{0}\in\Theta$. In a first step we will identify general conditions
on a prior $\Pi$ under which
the posterior $\Pi(f\in\cdot|Y,T)$ for $f$ concentrates
in a neighbourhood of $f_{0}$ with a certain rate of contraction $\xi_{\epsilon,\delta}$: We show for some constant $D>0$ the convergence
\begin{equation}
\Pi\big(f\in L^2(\mathcal D):\|f-f_{0}\|>D\xi_{\epsilon,\delta}|Y,T\big)\to0\label{eq:rateIntro}
\end{equation}
in $\P_{f_0,\theta_0}$-probability as $\epsilon$ and $\delta$ go to zero. This contraction result verifies that whole probability mass the posterior distribution is asymptotically located in a small ball around $f_0$ with radius of order $\xi_{\epsilon,\delta}\downarrow0$. Hence, draws from the posterior distribution will be close to the unknown function $f_0$ with high probability. This especially implies that the posterior mean and the posterior median are consistent estimators of the unknown function $f_0$. Interestingly, the difficulty to recover $f$ from $(Y,T)$ is same in all three above mentioned models. 

The proof of the contraction result follows general principles developed by \citet{ghosalEtAl2000}. The analysis of the posterior distribution requires to control both, the numerator in \eqref{eq:posteriorIntro} and the normalising constant. To find a lower bound for the latter, a so-called small ball probability condition is imposed ensuring that the prior puts some minimal weight in a neighbourhood of the truth. Given this bound, the contraction theorem can be shown by constructing sufficiently powerful tests for the hypothesis $H_{0}:f=f_{0}$ against the alternative $H_{1}:\|f-f_{0}\|>D\xi_{\epsilon,\delta}$ for the constant $D>0$ from \eqref{eq:rateIntro}. To find the test, we follow \citet{gineNickl2011}
and use a plug-in test based on a frequentist estimator. This estimator obtained by the Galerkin projection method, as proposed in \citep{efromovichKoltchinskii2001,HoffmannReiss2008} for the Model C.

The main difficulty is that without structural assumptions on $\Theta$,
e.g. if $\Theta=\ell_{2}$, an infinite dimensional nuisance parameter
$\theta$ cannot be consistently estimated. We thus cannot expect
a concentration of $\Pi(\theta\in\cdot|Y,T)$. Why should then $\Pi(K_{\theta}f\in\cdot|Y,T)$
concentrate around the truth? Fortunately, $K_{\theta_{0}}f_{0}$
is regular, such that a finite dimensional projection suffices to
reconstruct $f_{0}$ with high accuracy. Under the reasonable assumption
that the projection of $K_{\theta_{0}}$ depends only on a finite
dimensional projection $P\theta_{0}$ of $\theta_{0}$, we can indeed
estimate $f_{0}$ without estimating the full $\theta_{0}$. Similarly,
we show in the Bayesian setting that a concentration of this finite
dimensional projection $P\theta$ is sufficient resulting in a small
ball probability condition depending only on $f$ and $P\theta$.

The conditions of the general result are verified in the mildly ill-posed
case and in the severely ill-posed case, assuming some Sobolev regularity
of $f_{0}$. We use a truncated product prior of $f$
and a product prior on $\theta$. Choosing the truncation level $J$
of prior in an optimal way, the resulting contraction rates coincide
with the minimax optimal rates which are known in several models up
to logarithmic factors. These rates are indeed the same as for the
known parameter case, cf. \citet{ray2013}, if $\delta=\mathcal{O}(\epsilon)$. 

Since the optimal level $J$ depends on the unknown regularity $s$
of $f_{0}$, a data-driven procedure to select $J$ is desirable.
There are basically two ways to tackle this problem. Setting a hyper
prior on $s$, a hierarchical Bayes procedure could be considered.
Alternatively, although not purely Bayesian, we can try to select
some $\hat{J}$ empirically from the observations $Y,T$ and then
use this $\hat{J}$ in the Bayes procedure. Both possibilities are
only rarely studied for inverse problems. Using known operators, the
few articles on this topic include \citet{knapikEtAl2016} considering
both constructions with a Gaussian prior on $f$ and \citet{ray2013}
who has considered a sieve prior which could be interpreted as hierarchical
Bayes procedure. We will follow the second path and choose $J$ empirically
using Lepski's method \citep{lepski1990} which yields an easy to
implement procedure (note that \citep{knapikEtAl2016} used a maximum
likelihood approach to estimate $s$). We prove that the final adaptive
procedure attains the same rate as the optimized non-adaptive method.

This paper is organised as follows: The posterior distribution is
derived in Section~\ref{sec:posterior}. The general contraction
theorem is presented in Section~\ref{sec:ContractionRates}. In Section~\ref{sec:productPrior}
specific rates for Sobolev regular functions $f$ in the mildly and
the severely case are determined using a truncated product prior.
An adaptive choice of the truncation level is constructed in Section~\ref{sec:Adaptation}.
In Section~\ref{sec:Applications} we discuss the implementation
of the Bayes method using a Markov chain Monte Carlo algorithm and
illustrate the method in two numerical examples. All proofs are postponed
to Section~\ref{sec:Proofs}.

\section{Setting and posterior distribution\label{sec:posterior}}

Let us fix some notation: $\langle\cdot,\cdot\rangle$ and $\|\cdot\|$
denote the scalar product and the norm of $L^{2}(\mathcal Q)$ or $\Theta$. We
write $x\lesssim y$ if there is some universal constant $C>0$ such
that $x\le Cy$. If $x\lesssim y$ and $y\lesssim x$ we write $x\simeq y$. 
We recall that noise process $Z$ in (\ref{eq:obsY}) is the standard
iso-normal process, i.e., $\langle g,Z\rangle$ is $\mathcal{N}(0,\|g\|^{2})$-distributed
for any $g\in L^{2}(\mathcal Q)$ and covariances are given by
\[
\E[\langle Z,g_{1}\rangle\langle Z,g_{2}\rangle]=\langle g_{1},g_{2}\rangle\qquad\text{for all }g_{1},g_{2}\in L^{2}(\mathcal Q).
\]
We write $Z\sim\mathcal{N}(0,\Id)$. Note that $Z$ cannot
be realised as an element of $L^{2}(\mathcal Q)$, but only as an Gaussian process
$g\mapsto\langle g,Z\rangle$. 

The observation scheme (\ref{eq:obsY}) is equivalent to observing 
\[
\langle Y,g\rangle=\langle K_{\theta}f,g\rangle+\epsilon\langle Z,g\rangle\qquad\text{for all }g\in L^{2}(\mathcal Q).
\]
Choosing an orthonormal basis $(\phi_{k})_{k\ge1}$ of $L^2(\mathcal Q)$ with respect to the standard $L^2$-scalar product, we obtain the
series representation

\[
Y_{k}:=\langle Y,\phi_{k}\rangle=\langle K_{\theta}f,\phi_{k}\rangle+\epsilon Z_{k}
\]
for i.i.d. random variables $Z_{k}\sim\mathcal{N}(0,1),k\ge1$. Note
that the distribution of $(Z_{k})$ does not depend on $\theta$.
If $K_{\theta}$ is compact, it might be tempting to choose $(\phi_{k})$
from the singular value decomposition of $K_{\theta}$ simplifying
$\langle K_{\theta}f,\phi_{k}\rangle$. However, such a basis of eigenfunctions
will in general depend on the unknown $\theta$ and thus cannot be
used. Since $(Z_{k})_{k\ge1}$ are i.i.d., the distribution of the
vector $(Y_{k})_{k\ge1}$ is given by 
\[
\P_{\theta,f}^{Y}=\bigotimes_{k\ge1}\mathcal{N}\big(\langle K_{\theta}f,\phi_{k}\rangle,\epsilon^{2}\big).
\]
By Kakutani's theorem, cf. \citet[Theorem 2.7]{daPrato2006}, $\P_{\theta,f}$
is equivalent to the law $\P_{0}^{Y}=\bigotimes_{k\ge1}\mathcal{N}(0,\epsilon^{2})$
of the white noise $\epsilon Z$. Writing $\langle K_{\theta}f,Z\rangle:=\sum_{k\ge1}\langle K_{\theta}f,\phi_{k}\rangle Z_{k}$
with some abuse of notation, since $Z$ is not in $L^{2}(\mathcal Q)$, we obtain
the density
\[
\frac{\d\P_{\theta,f}^{Y}}{\d\P_{0}^{Y}}=\exp\Big(\frac{1}{\epsilon}\langle K_{\theta}f,Z\rangle-\frac{1}{2\epsilon^{2}}\sum_{k\ge1}\langle K_{\theta}f,\phi_{k}\rangle^{2}\Big)=\exp\Big(\frac{1}{\epsilon^{2}}\langle K_{\theta}f,Y\rangle-\frac{1}{2\epsilon^{2}}\|K_{\theta}f\|^{2}\Big),
\]
where we have used $Y_{k}=\epsilon Z_{k}$ under $\P_{0}^{Y}$ for
the second equality. 

Since any continuous operator $K_{\theta}$ can be described
by the infinite matrix $(\langle K_{\theta}\phi_{j},\phi_{k}\rangle)_{j,k\ge1}$,
we may assume with loss of generality that $\Theta\subset\ell^{2}$. The distribution of $T$ is then similarly given
by $\P_{\theta}^{T}=\bigotimes_{k\ge1}\mathcal{N}(\theta_{k},\delta^{2})$
being equivalent to $\P_{0}^{T}=\bigotimes_{k\ge1}\mathcal{N}(0,\delta^{2})$.
Writing $T=(T_{k})_{k\ge1}$ and $\langle\theta,T\rangle=\sum_{k\ge1}\theta_{k}T_{k}$,
we obtain the density 
\[
\frac{\d\P_{\theta}^{T}}{\d\P_{0}^{T}}=\exp\Big(\frac{1}{\delta^{2}}\langle\theta,T\rangle-\frac{1}{2\delta^{2}}\|\theta\|^{2}\Big).
\]
Therefore, the likelihood of the observations $(Y,T)$ with respect
to $\P_{0}^{Y}\otimes\P_{0}^{T}$ is given by
\begin{equation}
\frac{\d\P_{\theta,f}^{Y}\otimes\P_{\theta,f}^{T}}{\d\P_{0}^{Y}\otimes\P_{0}^{T}}=\exp\Big(\frac{1}{\epsilon^{2}}\langle K_{\theta}f,Y\rangle-\frac{1}{2\epsilon^{2}}\|K_{\theta}f\|^{2}+\frac{1}{\delta^{2}}\langle\theta,T\rangle-\frac{1}{2\delta^{2}}\|\theta\|^{2}\Big).\label{eq:likelihoodRatio}
\end{equation}
Applying a prior $\Pi$ on the parameter $(f,\theta)\in L^{2}(\mathcal D)\times\Theta$,
we obtain the posterior distribution
\begin{equation}
\Pi(B|Y,T)=\frac{\int_{B}e^{\epsilon^{-2}\langle K_{\theta}f,Y\rangle-(2\epsilon^{2})^{-1}\|K_{\theta}f\|^{2}+\delta^{-2}\langle\theta,T\rangle-(2\delta^{2})^{-1}\|\theta\|^{2}}\d\Pi(f,\theta)}{\int_{L^{2}\times\Theta}e^{\epsilon^{-2}\langle K_{\theta}f,Y\rangle-(2\epsilon^{2})^{-1}\|K_{\theta}f\|^{2}+\delta^{-2}\langle\theta,T\rangle-(2\delta^{2})^{-1}\|\theta\|^{2}}\d\Pi(f,\theta)},\qquad B\in\mathcal{B},\label{eq:posteriorY}
\end{equation}
with the Borel-$\sigma$-algebra $\mathcal{B}$ on $L^{2}(\mathcal Q)\times\Theta$.
Under the frequentist assumption that $Y$ and $T$ are generated
under some $f_{0}$ and $\theta_{0}$, we obtain the
representation
\begin{equation}
\Pi(B|Y,T)=\frac{\int_{B}p_{f,\theta}(Z,W)\d\Pi(f,\theta)}{\int_{L^{2}\times\Theta}p_{f,\theta}(Z,W)\d\Pi(f,\theta)},\qquad B\in\mathcal{B},\label{eq:posterior}
\end{equation}
for
\[
p_{f,\theta}(z,w):=\exp\Big(\frac{1}{\epsilon}\langle K_{\theta}f-K_{\theta_{0}}f_{0},z\rangle-\frac{1}{2\epsilon^{2}}\|K_{\theta}f-K_{\theta_{0}}f_{0}\|^{2}+\frac{1}{\delta}\langle\theta-\theta_{0},w\rangle-\frac{1}{2\delta^{2}}\|\theta-\theta_{0}\|^{2}\Big)
\]
corresponding to the density of $\P_{\theta,f}^{Y}\otimes\P_{\theta}^{T}$ with respect to $\P_{\theta_0,f_0}^{Y}\otimes\P_{\theta_0}^{T}$. 

Note that even if a Gaussian prior is chosen, the posterior distribution
is in general not Gaussian, since $\theta\mapsto K_{\theta}$ might
be non-linear. Hence, the posterior distribution cannot be explicitly
calculated in most cases, but has to be approximated by an MCMC algorithm,
see for instance \citet{tierney1994} and Section~\ref{sec:Applications}. 

\section{Contraction rates\label{sec:ContractionRates}}

For simplicity we throughout suppose $\mathcal{D}=\mathcal{Q}$ such that $L^{2}:=L^{2}(\mathcal{D})=L^{2}(\mathcal{Q})$
and assume $K_{\theta}$ to be self-adjoint. The general case is discussed in Remark~\ref{rem:notAdjoint}.

Taking a frequentist point of view, we assume that the observations
(\ref{eq:obsY}) and (\ref{eq:obsT}) are generated by some fixed
unknown $f_{0}\in L^{2}$ and $\theta_{0}\in\Theta$. As a first main
result the following theorem gives general conditions on the prior
which ensure a contraction rate for the posterior distribution from
(\ref{eq:posterior}) around the true $f_{0}$. 

Let $(\phi_{(j,l)}:j\in\mathcal{\mathcal{I}},l\in\mathcal{Z}_{j})$
be an orthonormal basis of $L^{2}$. We use here the double-index
notation which is especially common for wavelet bases, but also the
single-indexed notation is included if $\mathcal{Z}_{j}$ contains
only one element. For any index $k=(j,l)$ we write $|k|:=j$. Let
moreover $V_{j}=\Span\{\phi_{k}:|k|\le j\}$ be a sequence of approximation
spaces with dimensions $d_{j}\in\N$ associated to $(\phi_{k})$.
We impose the following compatibility assumption on $(\phi_{k})$:
\begin{assumption}
\label{ass:bases} There is some $m\in\N$ such that $\langle K_{\theta}\phi_{l},\phi_{k}\rangle=0$
for any $\theta\in\Theta$ if $\big||l|-|k|\big|>m$.
\end{assumption}

If $K_{\theta}$ is compact and admits an orthonormal basis of eigenfunction
$(e_{k})_{k\ge1}$ being independent of $\theta$, then this is assumption
is trivially satisfied for $(\phi_{k})=(e_{k})$ and $m=0$. On the
other hand this assumption allows for more flexibility for the considered
approximation spaces and can be compared to Condition 1 by \citet{ray2013}.
As a typical example, the possibly $\theta$ depended eigenfunctions
$(e_{\theta,k})$ of $K_{\theta}$ may be the trigonometric basis
of $L^{2}$ while $V_{j}$ are generated by band-limited wavelets.

Having $(\phi_{k})$ and thus $V_{j}$ fixed, we write $\|A\|_{V_{j}\to V_{j}}:=\sup_{v\in V_{j},\|v\|=1}\|Av\|$
for the operator norm for any bounded linear operator $A\colon V_{j}\to V_{j}$
where $V_{j}$ is equipped with the $L^{2}$-norm. We denote by $P_{j}$
the orthognal projection of $L^{2}$ onto $V_{j}$ and define the
operator
\[
K_{\theta,j}:=P_{j}K_{\theta}|_{V_{j}}
\]
as restriction of $K_{\theta}$ to an operator from $V_{j}$ to $V_{j}$.
Note that $K_{\theta,j}$ is given by the finite dimensional matrix
$(\langle K\phi_{k},\phi_{l}\rangle)_{|k|\le j,|l|\le j}\in\R^{d_{j}\times d_{j}}$.
\begin{assumption}
\label{ass:operator}Let $K_{\theta,j}$ depend only on a finite dimensional
projection $P_{j}\theta:=(\theta_{1},\dots,\theta_{l_{j}})$ of $\theta\in\Theta$
for some integer $1\le l_{j}\le d_{j}^{2}$, $j\in\mathcal{I}$. Moreover, let
$K_{\theta,j}$ be Lipschitz continuous with respect to $\theta$
in the following sense: 
\[
\|K_{\theta,j}-K_{\theta',j}\|_{V_{j}\to V_{j}}\le L\|P_{j}(\theta-\theta')\|_{j}\quad\text{for all }\theta,\theta'\in\Theta
\]
where $L>0$ is a constant being independent of $j,\theta,\theta'$ and where $\|\cdot\|_{j}$ is a norm on $P_{j}\Theta$.
We suppose that the norm $\|\cdot\|_{j}$ satisfies $\P_{\theta}\big(\|P_{j}W\|_{j}>C(\kappa+\sqrt{d_{j}})\big)\le\exp(-c\kappa^{2}).$
\end{assumption}

Although projections on $L^{2}$ and on $\Theta$ are both denoted
by $P_{j}$, it will be always clear from the context which is used
such that this abuse of notation is quite convenient. Since $K_{\theta,j}$
is fully described by a $d_{j}\times d_{j}$ matrix, we naturally
have the upper bound $l_{j}\le d_{j}^{2}$. Let us illustrate the
previous assumptions in the models A,B and C from the introduction:
\newpage
\begin{example}~
  \begin{enumerate}
  \item In Model A we have a finite dimensional parameter space $\Theta\subset\R^{p}$ with fixed $p\in\N$. Assumption~\ref{ass:bases} is, for instance, satisfied if $K_{\theta}f=g_{\theta}\ast f$
  is a convolution operator with a kernel $g_{\theta}$ whose Fourier
  transform has compact support and if we choose a band-limited wavelet basis. Note that in this case we do not have know the SVD of $K_{\theta}$. For Assumption~\ref{ass:operator} we may choose $P_{j}=\Id$ and
  $\|\cdot\|_{j}=|\cdot|$ as the Euclidean distance on $\R^{p}$ leading
  to the Lipschitz condition $\|K_{\theta,j}-K_{\theta',j}\|_{V_{j}\to V_{j}}\le L|\theta-\theta'|$.
  Then $\P_{\theta}\big(|W|>\sqrt{p}\kappa\big)\le2pe^{-\kappa^{2}/2}$
  follows from the Gaussian concentration of $W$.
  \item In Model B let $K_{\theta}$ be compact and let $(e_{i})_{i\ge1}$
  be an orthonormal basis consisting of eigenfunctions with corresponding
  eigenvectors $(\rho_{\theta,i})_{i\ge1}$ and let $(\phi_{k})$ be a wavelet
  basis fulfilling $d_{j}\simeq2^{dj}$. Then Assumption~\ref{ass:bases}
  is satisfied if there is some $C>0$ such that $\langle e_{i},\phi_{k}\rangle\neq0$
  only if $C^{-1}2^{dk}\le i\le C2^{dk}$. Since then $\langle e_{k},v\rangle=0$
  for any $v\in V_{j}$ if $k\ge C2^{dj}$, we moreover have for any
  $v\in V_{j}$
  \begin{align*}
  \big\|(K_{\theta,j}-K_{\theta',j})v\big\|^{2} & =\big\| P_{j}\sum_{i\ge1}(\rho_{\theta,i}-\rho_{\theta',i})\langle e_{i},v\rangle e_{i}\big\|^{2}\\
  & \le\sup_{i\le C2^{dj}}|\rho_{\theta,i}-\rho_{\theta',i}|^{2}\sum_{i\le C2^{dj}}\langle e_{i},v\rangle^{2}\le\sup_{i\le C2^{dj}}|\rho_{\theta,i}-\rho_{\theta',i}|^{2}\|v\|^{2}.
  \end{align*}
  We thus choose $l_{j}=C2^{dj}\simeq d_{j}$ and $\|\cdot\|_{j}$ as
  the supremum norm on $P_{j}\Theta$. Since $W_{k}$ are i.i.d. Gaussian,
  we have for some $c>0$
  \begin{align*}
  \P\big(\sup_{k\le C2^{dj}}|W_{k}|>\kappa+\sqrt{c\log d_{j}}\big) & \le2C2^{dj}e^{-(\kappa^{2}+c\log d_{j})/2}\le2Ce^{-\kappa^{2}/2}.
  \end{align*}
  Therefore, Assumption~\ref{ass:operator} is satisfied.
  \item In Model C the projected operators $K_{\theta,j}$ are given by $\R^{d_{j}\times d_{j}}$
  matrices. Assumption~\ref{ass:bases} is satisfied if and only if all $K_{\theta,j}$ are band matrices with some fixed bandwidth $m$ independent from $j$ and $\theta$. To verify Assumption~\ref{ass:operator}, $\|\cdot\|_{j}$ can be chosen as the operator norm or
  spectral norm of these matrices. The Lipschitz condition is then obviously
  satisfied. Moreover $P_{j}WP_{j}$ is a $\R^{d_{j}\times d_{j}}$
  random matrix where all entries are i.i.d. $\mathcal{N}(0,1)$ random
  variables. A standard result for i.i.d. random matrices is the bound
  $\E[\|P_{j}WP_{j}\|_{V_{j}\to V_{j}}]\lesssim\sqrt{d_{j}}$ for the
  operator norm, cf. \cite[Cor. 2.3.5]{tao2012}. Together with the
  Borell-Sudakov-Tsirelson concentration inequality for Gaussian processes,
  cf. \cite[Thm. 2.5.8]{gineNickl2016}, we immediately obtain the
  concentration inequality in Assumption~\ref{ass:operator}. 
  \end{enumerate}
\end{example}

Finally, the degree of ill-posedness of $K_{\theta}$ can be quantified
by the smoothing effect of the operator:
\begin{assumption}
\label{ass:ill-posed}For a decreasing sequence $(\sigma_{j})_{j}\subset(0,\infty)$
and some constant $Q>0$ let the operator $K_{\theta}$ satisfy $Q^{-1}\sum_{k}\sigma_{|k|}\langle f,\phi_{k}\rangle^{2}\le\langle K_\theta f,f\rangle\le Q\sum_{k}\sigma_{|k|}\langle f,\phi_{k}\rangle^{2}$
for all $f\in L^{2}$ and $\theta\in\Theta$.
\end{assumption}
Note that Assumptions~\ref{ass:bases} and \ref{ass:ill-posed} with $\sigma_j\downarrow0$ imply that $K_\theta$ is compact, because it can be approximated by the operator sequence $K_\theta P_j$ having finite dimensional ranges. The rate of the decay of $\sigma_j$ will determine the degree of ill-posedness of the inverse problem. If $\sigma_j$ decays polynomially or exponentially, we obtain a mildly or severely illposed problem, respectively.

Recall that the nuisance parameter $\theta$ cannot be consistently
estimated without additional assumptions. Therefore, we study the
contraction rate of the marginal posterior distribution $\Pi(f\in\cdot|Y,T)$.
While we allow for a general prior $\Pi_{f}$ on $L^{2}$ for $f$,
we will use a product prior on $\theta$. For densities $\beta_{k}$
on $\R$ we thus consider prior distributions of the form
\begin{equation}
\d\Pi(f,\theta)=\d\Pi_{f}(f)\otimes\bigotimes_{k\ge1}\beta_{k}(\theta_{k})\d\theta_{k}.\label{eq:prior}
\end{equation}
\begin{thm}
\label{thm:contraction}Consider the model (\ref{eq:obsY}), (\ref{eq:obsT})
generated by some $f_{0}\in L^{2}$ and $\theta_{0}\in\Theta$ with
$\epsilon=\epsilon_{n}\to0$ and $\delta=\delta_{n}\to0$ for $n\to\infty$,
respectively, and let Assumptions \ref{ass:bases}, \ref{ass:operator}
and \ref{ass:ill-posed} be satisfied. Let $\Pi_{n}$ be a sequence
of prior distributions of the form (\ref{eq:prior}) on the Borel-$\sigma$-algebra
on $L^{2}\times\Theta$. Let $(\kappa_{n}),(\xi_{n})$ two positive
sequences converging to zero and $(j_{n})$ a sequence of integers
with $j_{n}\to\infty$. Suppose $\kappa_{n}/(\epsilon_{n}\vee\delta_{n})\to\infty$
as $n\to\infty$ as well as 
\[
d_{j_{n}}\le c_{1}\frac{\kappa_{n}^{2}}{(\epsilon_{n}\vee\delta_{n})^{2}},\quad\frac{\kappa_{n}}{\sigma_{j_{n}}}\le c_{2}\xi_{n}\quad\text{and}\quad\frac{\delta_{n}}{\sigma_{j_{n}}}\sqrt{d_{j_{n}}}\to0
\]
for constants $c_{1},c_{2}>0$ and all $n\ge0.$ Suppose $f_{0}$
satisfies $\|f_{0}\|\le R$ and $\|f_{0}-P_{j_{n}}(f_{0})\|\le C_{0}\xi_{n}$
for some $R,C_{0}>0$. Let $\F_{n}\subset\{f\in L^{2}:\|f-P_{j_{n}}f\|\le C_{0}\xi_{n}\}$
be a sequence and $C_{1}>0$ such that
\begin{equation}
\Pi_{n}(L^{2}\setminus\mathcal{F}_{n})\le e^{-(C_{1}+4)\kappa_{n}^{2}/(\epsilon_{n}\vee\delta_{n})^{2}}.\label{eq:smallBias}
\end{equation}
Moreover assume for sufficiently large $n$ 
\begin{multline}
\Pi_{n}\Big((f,\theta)\in V_{j_{n}}\times\Theta:\\
\qquad\frac{\|P_{j_{n}+m}(K_{\theta}f-K_{\theta_{0}}f_{0})\|^{2}}{\epsilon_{n}^{2}}+\frac{\|P_{j_{n}+m}(\theta-\theta_{0})\|^{2}}{\delta_{n}^{2}}\le\frac{\kappa_{n}^{2}}{(\epsilon_{n}\vee\delta_{n})^{2}}\Big)\ge e^{-C_{1}\kappa_{n}^{2}/(\epsilon_{n}\vee\delta_{n})^{2}}.\label{eq:smallBall}
\end{multline}
Then there exists a finite constant $D>0$ such that the posterior
distribution from (\ref{eq:posterior}) satisfies
\begin{equation}
\Pi_{n}(f\in V_{j_{n}}:\|f-f_{0}\|>D\xi_{n}|Y,T)\to0\label{eq:rate}
\end{equation}
as $n\to\infty$ in $\P_{f_{0},\theta_{0}}$-probability.
\end{thm}

Theorem~\ref{thm:contraction} states that the posterior distribution $\Pi(f\in\cdot|Y,T)$ is consistent and concentrates asymptotically its whole probability mass in a ball around the true $f_0$ with decaying radius $D\xi_n\downarrow0$, that is, the posterior ``contracts to $f_0$'' with the rate $\xi_n$. This result is similarly to \citet[Theorem 2.1]{ray2013} who has
proven a corresponding theorem for known operators. However, the contraction
rate is now determined by the maximum $\epsilon\vee\delta$ instead
of $\epsilon$, which is natural in view of the results by \citet{HoffmannReiss2008}
who have included the case $\delta>\epsilon$ in their frequentist analysis. 

To gain some intuition on the interplay between $\kappa_n$ and the noise level $\epsilon_n\vee\delta_n$, let us set for simplicity $m=0$ in Assumption~\ref{ass:bases} and $\epsilon_n=\delta_n$. Using Assumption~\ref{ass:ill-posed} (with Lemma~\ref{lem:galerkin}) and Assumption~\ref{ass:operator}, we then can decompose
\begin{align*}
  \|f-f_0\|&\le\|P_{j_n}f_0-f_0\|+\|f-P_{j_n}f_0\|\\
  &\lesssim \|P_{j_n}f_0-f_0\|+\sigma_{j_n}^{-1}\|K_{\theta}f-K_{\theta}P_{j_n}f_0\|\\
  &\le\|P_{j_n}f_0-f_0\|+\sigma_{j_n}^{-1}\|K_{\theta}f-K_{\theta_0}P_{j_n}f_0\|+\sigma_{j_n}^{-1}L\|P_{j_n}(\theta-\theta_0)\|_j\|f_0\|
\end{align*}
The first term in the last line is the approximation error being bounded by $\xi_n$. It corresponds to the classical bias. Indeed, the prior sequence $\Pi_n$ is concentrated on a subset of $\{f:\|f-P_{j_n} f\|\le C_0\xi_n\}$ due to \eqref{eq:smallBias} such that the projection of $f$ to the level $j_n$ serves as reference measure for the prior and the deterministic error remains bounded by $\xi_n$. The  last two terms in the previous display correspond to the stochastic errors in $f$ and $\theta$ and are of the order $\kappa_{n}/\sigma_{j_{n}}$ owing the the minimal spread of $\Pi_n$ imposed by the small ball probability condition \eqref{eq:smallBall}. In particular, we recover the ill-posedness of the inverse problem due to $\sigma_{j_{n}}\to0$ in the denominator. To obtain the best possible contraction rate, we need to choose $j_{n}$ in way that ensures that $\xi_{n}$ is close to $\kappa_{n}/\sigma_{j_{n}}$, i.e., we will balance the deterministic and the stochastic error. The conditions on the dimension $d_{j_n}$ are mild technical assumptions.

The crucial small ball probability assumption \eqref{eq:smallBall} ensures that the prior sequence $\Pi_n$ has some minimal mass in a neighbourhood of the underlying $f_0$ and $\theta_0$. The distance from $(f_0,\theta_0)$ is measured in a (semi-)metric which reflects the structure of our inverse problem. If $\epsilon_n=\delta_n$, it would be sufficient if $\|K_\theta f-K_{\theta_0} f_0\|$ and $\|\theta-\theta_0\|$ are smaller than $\kappa_n$. However, condition \eqref{eq:smallBall} is more subtle. Firstly, the maximum of $\epsilon$
and $\delta$ on the right-hand side within the probability introduces
some difficulties. The prior has to weight a smaller neighbourhood
of $K_{\theta_{0}}f_{0}$ or $\theta_{0}$, respectively, depending
on whether $\epsilon$ is smaller than $\delta$ or the other way
around. If, for instance, $\epsilon<\delta$ the contraction rate
is determined by $\delta$ but the prior has to put enough probability
to the smaller $\epsilon$-ball around $K_{\theta_{0}}f_{0}$. We
see such effects also in the construction of lower bounds, cf. \citep{HoffmannReiss2008},
where we may have in the extreme case a $\delta$ distance between
$f$ and $f_{0}$ while $K_{\theta_{0}}f_{0}=K_{\theta}f$. Secondly,
(\ref{eq:smallBall}) depends only on finite dimensional projections
of both $K_{\theta}f$ and $\theta$. This is particularly important
as we do not assume any regularity conditions on $\theta$ such that
we cannot expect the projection remainder $(\Id-P_{j+m})\theta$ to
be small. 

To allow for this relaxed small ball probability condition,
the contraction rate is restricted to the set $V_{j}$. The result
can be extended to $L^{2}$ by appropriate constructions of the prior,
in particular, if the support of $\Pi_{n}$ is contained in $V_{j}$
we can immediately replace $V_{j}$ by $L^{2}$ in (\ref{eq:rate}).
Another possibility are general product prior if the basis is chosen
according to the singular value decomposition of $K_{\theta}$. 

To prove Theorem~\ref{thm:contraction}, we use the techniques by \citet[Thm. 2.1]{ghosalEtAl2000}, cf. also \citep[Thm. 7.3.5]{gineNickl2016}. A main step is the construction
of tests for the testing problem
\[
H_{0}:f=f_{0}\qquad\text{vs.}\qquad H_{1}:f\in\mathcal{F}_{n},\|f-f_{0}\|\ge D\xi_{n}.
\]
To this end, we first study a frequentist estimator of $f$ which
then allows to construct a plug in test as proposed by \citet{gineNickl2011}.

The natural estimator for $\theta$ is $T$ itself. In order to estimate
$f$, we use a linear Galerkin method based on the perturbed operator
$K_{T}$ similar to the approaches in \citep{efromovichKoltchinskii2001,HoffmannReiss2008}.
We thus aim for a solution $\hat{f}_{\epsilon,\delta}\in V_{j}$ to
\begin{equation}\label{eq:galerkin}
\langle K_{T}\hat{f}_{\epsilon,\delta},v\rangle=\langle Y,v\rangle\qquad\text{for all }v\in V_{j}.
\end{equation}
Choosing $v\in\{\phi_{k}:|k|\le j\}$, we obtain a system of linear equations
depending only on the projected operator $K_{T,j}$. There is a unique
solution if $K_{T,j}$ is invertible. Noting that for the unperturbed
operator $K_{\theta,j}$ Assumption~\ref{ass:ill-posed} implies
$\|K_{\theta,j}^{-1}\|_{V_{j}\to V_{j}}\le Q\sigma_{j}^{-1}$ (cf.
Lemma~\ref{lem:galerkin} below), we set
\begin{equation}
\hat{f}_{j}:=\begin{cases}
K_{T,j}^{-1}P_{j}Y, & \text{if }\|K_{T,j}^{-1}\|_{V_{j}\to V_{j}}\le\tau/\sigma_{j},\\
0, & \text{otherwise},
\end{cases}\label{eq:estimator}
\end{equation}
for a projection level $j$ and a cut-off parameter $\tau>0$. Adopting
ideas from \citep{gineNickl2011,HoffmannReiss2008}, we obtain the
following non-asymptotic concentration result for the estimator $\hat{f}_{j}$. 
\begin{prop}
\label{prop:concHatF}Let $j\in\N$, $\kappa>0$ such that $d_{j}\le C_{1}\kappa^{2}/(\epsilon\vee\delta)^{2}$
for some $C_{1}>0$. Under Assumptions~\ref{ass:operator} and \ref{ass:ill-posed}
there are constants $c,C>0$ such that, if $\delta\sigma_{j}^{-1}(\kappa+\sqrt{d_{j}})\le c\frac{\tau-Q}{\tau Q}$
and $\tau>Q$, then $\hat{f}_{j}$ from (\ref{eq:estimator}) fulfils
\[
\P_{f,\theta}\Big(\|\hat{f}_{j}-f\|\ge C\sigma_{j}^{-1}(\|f\|\vee1)\kappa+\|f-P_{j}f\|\Big)\le3e^{-\kappa^{2}/(\epsilon\vee\delta)^{2}}.
\]
\end{prop}

Note that some care will be needed to analyse the above mentioned
tests since also the stochastic error term $\sigma_{j}^{-1}(\|f\|\vee1)\kappa$
depends on the unknown function $f$ and, for instance, a Gaussian
prior on $f$ will not sufficiently concentrate on a fixed ball $\{f\in L^{2}:\|f\|\le R\}$. 

\begin{rem}\label{rem:notAdjoint}
  While the assumption that $K_\theta$ is self-adjoint simplifies the analysis and the presentation of our approach, the methodology can be generalised to general compact operators $K_\theta$. In this case Assumption~\ref{ass:ill-posed} should be replaced by the assumption $\|K_\theta f\|^2\simeq\sum_k\sigma_{|k|}^2\langle f,\phi_k\rangle^2$ which is consistent with the original condition, cf. Remark~\ref{rem:Kf2}. The Galerkin projection method \eqref{eq:galerkin} can then be generalised to solve
  \[
    \langle K_T^*K_{T}\hat{f}_{\epsilon,\delta},v\rangle=\langle Y,K_Tv\rangle\qquad\text{for all }v\in V_{j},
  \]
  cf. \citet[Appendix A]{CohenEtAl2004}. This modified estimator should have a similar behaviour as above such that we can construct the tests which we needed to prove Theorem~\ref{thm:contraction}. The rest of the proof of the contraction theorem and the subsequent results would remain as before.
\end{rem}

\section{A truncated product prior and the resulting rates\label{sec:productPrior}}

For the ease of clarity we fix a ($S$-regular) wavelet basis $(\phi_{k})_{k\in\{-1,0,1,\dots\}\times\Z}$
of $L^{2}$ with the associated approximation spaces $V_{j}=\Span\{\phi_{k}:|k|\le j\}$.
We write $|k|=|(j,l)|=j$ as before. Investigating a bounded
domain $\mathcal{D}\subset\R^{d}$, we have in particular $d_{j}\simeq2^{jd}$.
The regularity of $f$ will be measured in the Sobolev balls
\begin{equation}
H^{s}(R):=\Big\{ f\in L^{2}([0,1])\,:\,\|f\|_{H^{s}}^{2}:=\sum_{j=-1}^{\infty}2^{2sj}\sum_{l}\langle f,\phi_{j,l}\rangle^{2}\le R^{2}\Big\},\qquad s\in\R.\label{eq:hoelderball}
\end{equation}
We will use Jackson's inequality and the Bernstein inequality: For
$-S<s\le t<S$ and $f\in H^{s}$, $g\in V_{j}$ we have
\begin{equation}
\|(\Id-P_{j})f\|_{H^{s}}\lesssim2^{-j(t-s)}\|f\|_{H^{t}}\quad\text{and}\quad\|g\|_{H^{t}}\lesssim2^{j(t-s)}\|g\|_{H^{s}}.\label{eq:approx}
\end{equation}
\begin{rem}
\label{rem:FourierBasis}The subsequent analysis applies also to the
trigonometric as well as the sine basis in the case of periodic functions.
Considering more specifically $L_{per}^{2}([0,1])=\{f\in L^{2}([0,1]):f(0)=f(1)=0\}$,
we may set $\phi_{k}=\sqrt{2}\sin(\pi k\cdot)$ for $k\in\N$. Since
$\|f\|_{H^{s}}^{2}\simeq\sum_{k\ge1}j^{2s}\langle f,\phi_{k}\rangle^{2}$
holds for any $f\in L_{per}^{2}([0,1])$, it is then easy to see
that the inequalities (\ref{eq:approx}) are satisfied for $V_{j}=\Span\{\phi_{1},\dots,\phi_{j}\}$
if $2^{j}$ is replaced by $j$. Alternatively we may set $V_{j}=\Span\{\phi_{1},\dots,\phi_{2^{j}}\}$
which gives exactly (\ref{eq:approx}).
\end{rem}

For $\theta$ we use the product prior as in (\ref{eq:prior}) with
a fixed density $\beta_{k}=\beta$. For $f$ we also a apply a product
prior. More precisely, we take a prior $\Pi_{f}$ determined by the
random series
\[
f(x)=\sum_{|k|\le J}\tau_{|k|}\Phi_{k}\phi_{k}(x),\qquad x\in[0,1],
\]
for a sequence $(\tau_{j})_{j\ge-1}$, i.i.d. random coefficients
$\Phi_{k}$ (independent of $\theta_{k}$) distributed according to
a density $\alpha$ and a cut-off $J\in\N$. Hence,
\begin{equation}
\d\Pi(\theta,f)=\prod_{|k|\le J}\tau_{|k|}^{-d}\alpha(\tau_{|k|}^{-1}f_{k})\,\d f_{k}\cdot\prod_{k\ge1}\beta(\theta_{k})\,\d\theta_{k}.\label{eq:productPrior}
\end{equation}
Under appropriate conditions on the distributions $\alpha,\beta$
and on $J$ we will verify the conditions of Theorem~\ref{thm:contraction}. 
\begin{assumption}
\label{ass:densities}There are constants $\gamma,\Gamma>0$ such that the densities $\alpha$ and $\beta$ satisfy 
\[
\alpha(x)\wedge\beta(x)\ge\Gamma e^{-\gamma|x|^{2}}\qquad\text{for all }x\in\R.
\]
\end{assumption}

Assumption~\ref{ass:densities} is very weak and is satisfied for
many distributions with unbounded support, for example, Gaussian,
Cauchy, Laplace distributions or Student's $t$-distribution. Also uninformative priors where $\alpha$ or $\beta$ are constant are included. A consequence
of the previous assumption is that any random variable $\Phi$ with
probability density $\alpha$ (or $\beta$) satisfies
\begin{equation}
\P(|\Phi-x|\le\kappa)\ge\Gamma\int_{|y|\le \kappa}e^{-\gamma|y+x|^2}\d y\ge2\Gamma\kappa e^{-\gamma(|x|+\kappa)^{2}}\qquad\text{for all }\kappa>0,x\in\R.\label{eq:raysBound}
\end{equation}
This lower bound will be helpful
to verify the small ball probabilities (\ref{eq:smallBall}). 

To apply Theorem~\ref{thm:contraction}, we choose $J=j_{n}$ to
ensure that the support of $\Pi_{f}$ lies in $\{f\in\mathcal{F}:\|P_{j_{n}}(f)-f\|\le C_{1}\xi_{n}\}$.
Note that the optimal $j_{n}$ is not known in practice. We will discuss
the a data-driven choice of $J$ in Section~\ref{sec:Adaptation}.
Alternatively to truncating the random series for $f$, the small
bias condition could be satisfied if $(\tau_{j})$ decays sufficiently
fast and $\alpha$ has bounded support, as it is the case for uniform
wavelet priors. 

We start with the mildly ill-posed case imposing $\sigma_{j}=2^{-jt}$
for some $t>0$ in Assumption~\ref{ass:ill-posed}. In this case
the operators $K_{\theta}$ are naturally adapted to Sobolev scale,
since then $K_{\theta}\colon L^{2}\to H^{t}$ is continuous with $\|K_{\theta}f\|\lesssim\|f\|_{H^{t}}$, cf. Remark~\ref{rem:Kf2}. 
\begin{thm}
\label{thm:mildly}Let $\epsilon^{\eta}\lesssim\delta\lesssim\epsilon$
for some $\eta>1$ and let Assumptions~\ref{ass:bases}, \ref{ass:operator}
with $l_{j}\le2^{jd}$, Assumption~\ref{ass:ill-posed} with $\sigma_{j}=2^{-jt}$
for some $t>0$ as well as Assumption~\ref{ass:densities} be fulfilled.
Then the posterior distribution from (\ref{eq:posterior}) with prior
given by (\ref{eq:productPrior}) where $J$ is chosen such that $2^{J}=\big(\epsilon\log(1/\epsilon)\big)^{-2/(2s+2t+d)}$
and $\frac{c}{j}2^{-j(2s_{0}+d)}\le\tau_{j}^{2}\le2^{Cj}$ for constants
$c,C>0$ and some $0<s_{0}<s$ satisfies for any $f_{0}\in H^{s}(R)$
and $\theta_{0}\in\Theta_{0}$
\[
\Pi_{n}\Big(f\in L^{2}:\|f-f_{0}\|>D\big(\epsilon\log(1/\epsilon)\big)^{2s/(2s+2t+d)}\Big|Y,T\Big)\to0
\]
with some constant $D>0$ and in $\P_{f_{0},\theta_{0}}$-probability.
\end{thm}

\begin{rem}
This theorem is restricted to the case $\epsilon\gtrsim\delta$. However,
its proof reveals that in the special case where $m=0$, for instance,
if $(\phi_{k})$ are eigenfunctions, the condition $\epsilon^{\eta}\lesssim\delta\lesssim\epsilon$
can be weakened to $\log\delta\simeq\log\epsilon$, which also allows
for $\epsilon<\delta$. The second restriction is $l_{j}\le2^{jd}$
which is especially satisfied in the model $B$ of unknown eigenvalues
in the singular value decomposition of $K_{\theta}$. Larger $l_{j}$
could be incorporated if we put some structure on $\Theta$ which
allows for applying a different prior on $\theta$ with better concentration
of $P_{j}\theta$.
\end{rem}

The contraction rate coincides with the minimax optimal convergence
rate, as determined in \citep{cavalierHengartner2005,HoffmannReiss2008}
for specific settings of $\theta\mapsto K_{\theta}$, up to the logarithmic
term. The conditions on $\tau_{j}$ are very weak and allow for a
large flexibility in the choice of prior, particularly, a constant
$\tau_{j}=1$ for all $j$ is included. In contrast, the choice of
the cut-off parameter $J$ is crucial and depends on the regularity
$s$ of $f_{0}$ and the ill-posedness $t$ of the operator. 

In the severely ill-posed case the contraction rates deteriorates
to a logarithmic dependence on $\epsilon\vee\delta$ and coincide
again with the minimax optimal rate.
\begin{thm}
\label{thm:severely}Let $\log\epsilon\simeq\log\delta$ and let Assumptions~\ref{ass:bases},
\ref{ass:operator}, Assumption~\ref{ass:ill-posed} with $\sigma_{j}=\exp(-r2^{jt})$
for some $r,t>0$ as well as Assumption~\ref{ass:densities} be fulfilled.
Then the posterior distribution from (\ref{eq:posterior}) with prior
given by (\ref{eq:productPrior}) where $J$ is chosen such that $2^{J}=\big(-\frac{1}{2r}\log(\epsilon\vee\delta)\big)^{1/t}$
and $2^{-j(2s_0+t+d)}\le\tau_{j}^{2}\le\exp(C2^{jt})$ for a constant
$C>0$ satisfies for any $f_{0}\in H^{s}(R)$ and $\theta_{0}\in\Theta_{0}$
\[
\Pi_{n}\Big(f\in L^{2}:\|f-f_{0}\|>D\big(\log(\epsilon\vee\delta)^{-1}\big)^{-s/t}\Big|Y,T\Big)\to0
\]
with some constant $D>0$ and in $\P_{f_{0},\theta_{0}}$-probability.
\end{thm}

\section{Adaptation via empirical Bayes\label{sec:Adaptation}}

We saw above that the choice of the projection level $J$ of the prior
depends on the unknown regularity $s$ (and the ill-posedness $t$)
in order to achieve the optimal rate. We will now discuss how $J$
can be chosen purely data-driven resulting in an empirical Bayes procedure
that adapts on $s$. Noting that choice of $J$ in Theorem~\ref{thm:severely}
is already independent of $s$, we focus on the mildly ill-posed case
and $\delta\lesssim\epsilon$.

The method is based on the observation that all conditions on the
level $j_{n}$ in Theorem~\ref{thm:contraction} are monotone (in
the sense that they are also satisfied for all $j$ smaller than the
optimal $j_{n}$) except for the bias condition on $\|f_{0}-P_{j_{n}}f_{0}\|\lesssim\xi_{n}$.
Given the optimal $J_{o}=j_{n}$, the so-called oracle, the result
in Theorem~\ref{thm:mildly} continues to hold for any, empirically
chosen $\hat{J}$ satisfying 
\[
\hat{J}\le J_{o}\qquad\text{and}\qquad\|f_{0}-P_{\hat{J}}f_{0}\|\lesssim\xi_{n}.
\]
To find $\hat{J}$, we use Lepski's method \citep{lepski1990} which
is generally known for these two properties. 

In Proposition~\ref{prop:concHatF} we saw that the variance of the
estimator $\hat{f}_{j}$ from (\ref{eq:estimator}) is of the order
$\epsilon^{2}d_{j}/\sigma_{j}^{2}=\epsilon^{2}2^{2jt+jd}$. For some
fixed lower bound $s_{0}$ on the regularity $s$ of $f_{0}\in H^{s}$
let 
\[
\mathcal{J}_{\epsilon}=\Big\lfloor\frac{\log\epsilon^{-1}}{(s_{0}+t+d/2)\log2}\Big\rfloor
\]
where $\lfloor x\rfloor$ denotes be the largest integer smaller than
$x$. The choice of $\mathcal{J}_{\epsilon}$ allows for applying
the concentration inequality from Propotion~\ref{prop:concHatF}
to all $\hat{f}_{j}$ with $j\le\mathcal{J}_{\epsilon}$. We then
choose
\[
\hat{J}:=\min\big\{ j\in\{1,\dots,\mathcal{J}_{\epsilon}\}:\|\hat{f}_{i}-\hat{f}_{j}\|\le\Delta\epsilon(\log\epsilon^{-1})^{2}2^{i(t+d/2)}\forall i>j\big\}
\]
for a constant $\Delta\in(0,1]$ which can be chosen by the practitioner.
The idea of the choice $\hat{J}$ is as follows: Starting with large
$j$ the projection estimator $\hat{f}_{j}$ has a small bias, but
a standard deviation of order $\epsilon2^{j(t+d/2)}$. Decreasing
$j$ reduces the variance while the bias increases. At the point where
there is some $i>j$ such that $\|\hat{f}_{i}-f_{0}\|+\|\hat{f}_{j}-f_{0}\|\ge\|\hat{f}_{i}-\hat{f}_{j}\|$
is larger than the order of the variance the bias starts dominating
the estimation error. At this point we stop lowering $j$ and select
$\hat{J}$.
\begin{thm}
\label{thm:adaptive}Let $\epsilon^{\eta}\lesssim\delta\lesssim\epsilon$
for some $\eta>1$ and let Assumptions~\ref{ass:bases}, \ref{ass:operator}
with $l_{j}\le2^{jd}$, Assumption~\ref{ass:ill-posed} with $\sigma_{j}=2^{-jt}$
for some $t>0$ as well as Assumption~\ref{ass:densities} be fulfilled.
Then the posterior distribution from (\ref{eq:posterior}) with prior
given by (\ref{eq:productPrior}) with $\hat{J}$ instead of $J$
and $\frac{c}{j}2^{-j(2s_{0}+d)}\le\tau_{j}^{2}\le2^{Cj}$ for constants
$c,C>0$ and some $0<s_{0}<s$ satisfies for any $f_{0}\in H^{s}(R)$
and $\theta_{0}\in\Theta_{0}$
\[
\Pi_{n}\Big(f\in L_{2}:\|f-f_{0}\|>D(\log\epsilon^{-1})^{\chi}\epsilon^{2s/(2s+2t+d)}\Big|Y,T\Big)\to0
\]
with some constant $D>0$, $\chi=(4s+2t+d)/(2s+2t+d)$ and in $\P_{f_{0},\theta_{0}}$-probability.
\end{thm}

Note that the empirical Bayes procedure is adaptive with respect to
$s$ and the Sobolev radius $R$. Compared to Theorem~\ref{thm:mildly}
where the oracle choice for $J$ is used, we only lose a logarithmic
factor for adaptivity.

\section{Examples and Simulations\label{sec:Applications} }

\subsection{Heat equation with unknown diffusivity parameter\label{sec:Heat}}

To illustrate the previous theory, we consider the heat equation 
\begin{equation}
\frac{\partial}{\partial t}u(x,t)=\theta\frac{\partial^{2}}{\partial x^{2}}u(x,t),\qquad u(\cdot,0)=f,\qquad u(0,t)=u(1,t)=0\label{eq:heatEq}
\end{equation}
with Dirichlet boundary condition at $x=0$ and $x=1$ and some initial
value function $f\in L^{2}([0,1])$ satisfying $f(0)=f(1)=0$. Different
to \citep{knapikEtAl2013,ray2013} we take an unknown
diffusivity parameter $\theta>0$ into account. A solution to (\ref{eq:heatEq})
is observed at some time $t>0$
\begin{equation}
Y=u(\cdot,t)+\epsilon Z\label{eq:obs}
\end{equation}
with white noise $Z$ on $L^{2}([0,1])$. The aim is to recover $f$
from $Y$.

The solution $u(\cdot,t)$ depends linearly on $f$ via an operator
$K_{\theta}$ which is diagonalised by the sine basis $e_{k}=\sqrt{2}\sin(\pi k\cdot),k\ge1,$
of $L_{per}^{2}([0,1])$ building a system of eigenfunctions of the Laplace
operator. The corresponding eigenvalues of $K_{\theta}$ are given
by $\rho_{\theta,k}:=e^{-\theta\pi^{2}k^{2}t},k\ge1,$ and we obtain
the singular value decomposition 
\[
K_{\theta}f=\sum_{k\ge1}\langle f,e_{k}\rangle\rho_{\theta,k}e_{k}=\sum_{k\ge1}\langle f,e_{k}\rangle e^{-\theta\pi^{2}k^{2}t}\sqrt{2}\sin(\pi k\cdot).
\]

Note that $K_{\theta}$ depends on $\theta$ only via its eigenvalues
$\rho_{\theta,k}$ while the eigenfunctions and thus the considered
basis is independent of $\theta$. Moreover the dependence of $\rho_{\theta,k}$
on $\theta$ is non-linear. From the decay of the eigenvalues we see
that the resulting inverse problem is severely ill-posed with $\sigma_{j}=\exp(-\theta\pi^{2}tj^{2})$.
Since we can easily construct pairs $(\theta,f)$ and $(\theta',f')$
with $K_{\theta}f=K_{\theta'}f'$, the function $f$ is indeed not
identifiable only based on the observation $Y$ and we need the additional
observation $T=\theta+\delta W$ for some $W\sim\mathcal{N}(0,1)$.

Since the eigenfunctions are independent of $\theta$, we can choose
the basis $\phi_{k}=e_{k}$ thanks to Remark~\ref{rem:FourierBasis}.
We moreover apply the truncated product prior \eqref{eq:productPrior} with centered normal densities densities $\alpha$ and $\beta$ and fixed variances $\tau^{2}$ and $\sigma^{2}$. In our numerical example we set $t=0.1$,
\begin{equation}
f_{0}(x)=4x(1-x)(8x-5)\quad\text{and}\quad\theta_{0}=1\label{eq:f0}
\end{equation}
reproducing the same setting as considered in \citep{knapikEtAl2013},
but taking the unknown $\theta$ into account. The Fourier coefficients
of $f_{0}$ with respect to the sine series $\phi_{k}$ are given
by 
\[
f_{0,k}=\langle f_{0},\phi_{k}\rangle=\frac{8\sqrt{2}(13+11(-1)^{k})}{\pi^{3}k^{3}},\qquad k\ge1.
\]
By the decay of the coefficients, we have $f_{0,k}\in H^{s}$ for every
$s<5/2$.

To implement our Bayes procedure, we need to sample from the posterior
distribution which is not explicitly accessible. Fortunately, using
independent normal $\mathcal{N}(0,\tau^{2})$ priors on the coefficients
$f_{k}=\langle f,\phi_{k}\rangle$, we see from~(\ref{eq:posteriorY})
that at least the conditional posterior distribution of $f$ given
$\theta,Y,T$ can be explicitly computed as
\begin{equation}
\Pi(f\in\cdot|\theta,Y,T)=\bigotimes_{k\le J}\mathcal{N}\Big(\frac{\epsilon^{-2}\rho_{\theta,k}^{-1}}{\epsilon^{-2}+\rho_{\theta,k}^{-2}\tau^{-2}}Y_{k},\frac{\rho_{\theta,k}^{-2}}{\epsilon^{-2}+\rho_{\theta,k}^{-2}\tau^{-2}}\Big).\label{eq:conditionalPosterior}
\end{equation}
Profiting from this known conditional posterior distribution, we use
a Gibbs sampler to draw (approximately) from the unconditional posterior
distribution of $f$ given $Y,T$, cf. \citep{tierney1994}. Given
some initial $\theta^{(0)}$, the algorithm alternates between draws
from $f^{(i+1)}|\theta=\theta^{(i)},Y,T$ and $\theta^{(i+1)}|f=f^{(i+1)},Y,T$
for $i\in\N$. The second conditional distribution is not explicitly
given, due to the non-linear dependence of $\rho_{\theta,k}$ from
$\theta$. We apply a standard Metropolis-Hastings algorithm to approximate
the distribution of $\theta|f,Y,T$ using a random walk with $\mathcal{N}(0,v^{2})$
increments as proposal chain. A similar Metropolis-within-Gibbs method
has been used in \citep{knapikEtAl2016} in a comparable simulation
task. Using the sequence $(\theta^{(i)})_{i}$ from this algorithm,
the final Markov chain Monte Carlo (MCMC) approximation of $\Pi(f\in\cdot|Y,T)$
is then given by an average
\[
\frac{1}{M}\sum_{m=1}^{M}\Pi\big(f\in\cdot|\theta=\theta^{(B+m*l)},Y,T\big)
\]
for sufficiently large $B,M,l\in\N$, where we again profit from the
explicitly given conditional posterior distribution (\ref{eq:conditionalPosterior}).

\begin{figure}
  \includegraphics[width=0.5\textwidth]{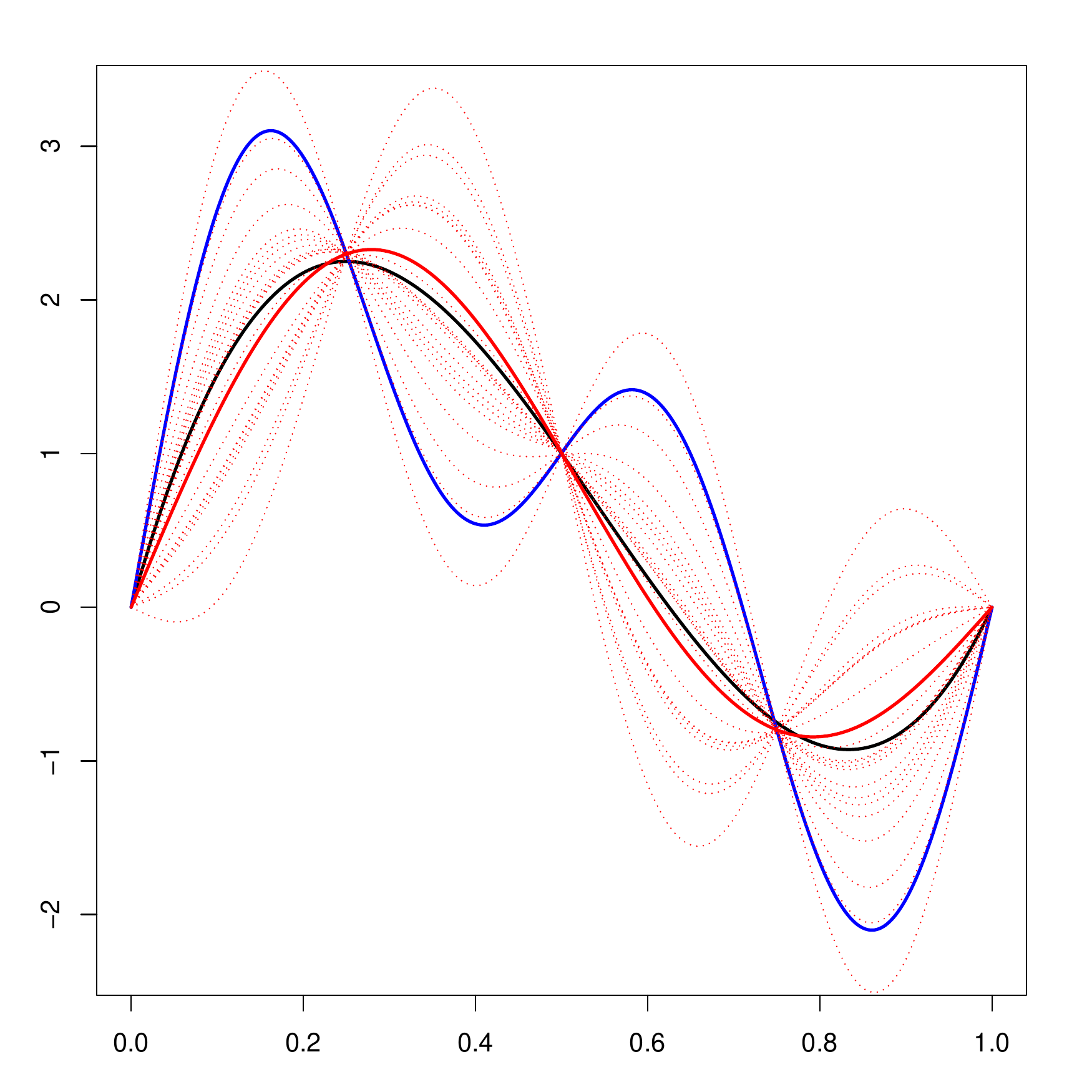}\includegraphics[width=0.5\textwidth]{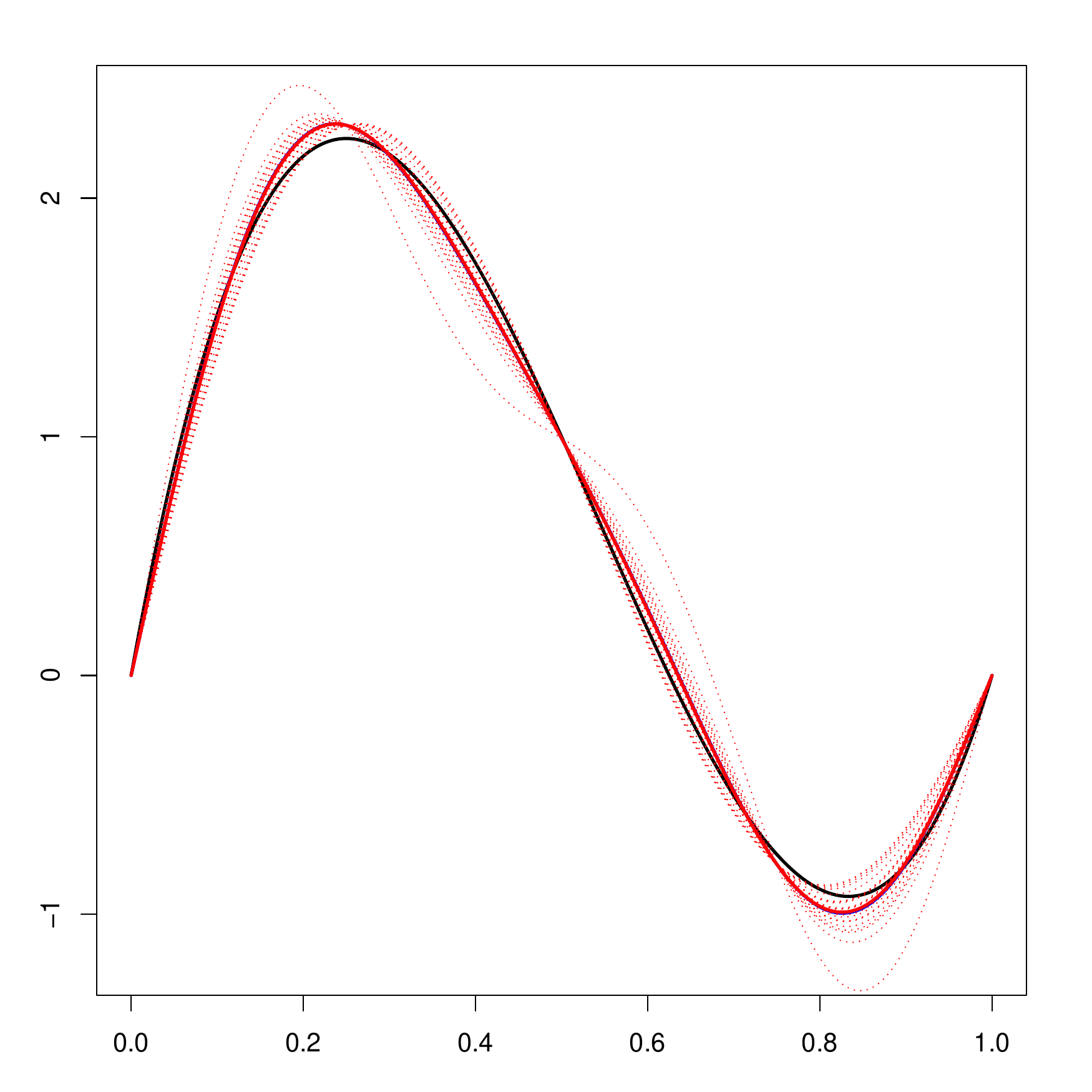}
  \caption{\label{fig:heat}Heat equation with unknown diffusivity parameter:
  True function (black), projection estimator (blue), posterior mean
  (red, solid) and 20 draws from the posterior distribution (red, dotted)
  with $\protect\epsilon=\delta=10^{-6}$ (left) and $10^{-8}$ (right). }
\end{figure}

Figure~\ref{fig:heat} shows the typical posterior mean and 20 draws
from the posterior distribution in a simulation using $\epsilon=\delta=10^{-6}$
and $10^{-8}$. In both cases the projection level is chosen as $J=4\simeq\sqrt{-\log(\epsilon)}$.
Especially for the smaller noise level, the common intersections of
all sampled functions are conspicuous. They reflect a quite low variance
of the posterior distribution in the first coefficients compared to
a relatively large variance already for $f_{4}$ due to the severe
ill-posedness, cf. (\ref{eq:conditionalPosterior}). 

As a reference estimator the Galerkin projector $\hat{f}_{J}$ from
(\ref{eq:estimator}) is plotted, too. We see that for $\epsilon=10^{-6}$
the posterior mean is much closer to the true function indicating an efficiency gain of the Bayesian procedure compared to the projection estimator. 
For $\epsilon=10^{-8}$ both estimators coincide almost perfectly. As shown by the theory, the figure
illustrates that the posterior distribution concentrates around the
truth for smaller noise levels. Monte Carlo simulations based on 500
iterations yield a root mean integrated squared error (RMISE) 0.3353
and 0.0512 for $\epsilon=10^{-6}$ and $\epsilon=10^{-8}$, respectively.
For the posterior mean of $\theta$ we observe a root mean squared
error of approximately $1.0\cdot10^{-6}$ and $9.7\cdot10^{-9}$,
respectively. Additionally, Table~\ref{tab:RMISE} reports the RMISE for several different combinations of the noise levels $\epsilon$ and $\delta$.

\begin{table}[t]\centering
  \begin{tabular}{c|ccc}
    $\delta$ \textbackslash\, $\epsilon$ & $10^{-4}$ & $10^{-6}$ & $10^{-8}$ \\ \hline
    $10^{-4}$ & 0.5728 & 0.3173 & 0.5656\\
    $10^{-6}$ & 0.5515 & 0.3353 & 0.0545\\
    $10^{-8}$ & 0.5548 & 0.3269 & 0.0512\\
  \end{tabular}
  \caption{RMISE for different values of $\epsilon$ and $\delta$.}\label{tab:RMISE}
\end{table}

\subsection{Deconvolution with unknown kernel}

Another example is the deconvolution problem occurring for instance
in image processing, cf. \citet{johnstoneEtAl2004}. The aim is to
recover some unknown 1-periodic function $f$ from the observations
\[
Y=K_{\theta}f+\epsilon Z\quad\text{with}\quad K_{\theta}f:=g_{\theta}\ast f:=\int_{0}^{1}f(\cdot-x)g_{\theta}(x)\d x
\]
where $g_{\theta}\in L^{2}_{per}$ is some $1$-periodic convolution kernel
(more general it might be a signed measure). Since the convolution
operator $K_{\theta}$ is smoothing, the inverse problem is ill-posed.
If the kernel $g_{\theta}$ is unknown, the problem is called blind deconvolution occurring in many applications \cite{BurgerScherzer2001,JustenRamlau2006,StueckEtAl2012}.
In a density estimation setting this problem as already been intensively
investigated, cf. \citep{dattnerEtAl2016,johannes2009, JohannesEtAl2011,neumann1997}
among others. However, the Bayesian perspective on this problem seem
not thoroughly studied. 

We consider the trigonometric basis 
\[
\phi_{0}=1,\quad\phi_{j,0}=\sqrt{2}\sin(2\pi j\cdot),\qquad\phi_{j,1}=\sqrt{2}\cos(2\pi j\cdot),\qquad j\in\N,
\]
with the corresponding approximation spaces $V_{J}=\Span(\phi_{j,l}:j\le J,l\in\{0,1\})$.
Assuming $g_{\theta}$ is symmetric, we have $\langle g_{\theta},\phi_{j,0}\rangle=0$
and
\begin{align*}
K_{\theta}\phi_{0}=\langle g_{\theta},\phi_{0}\rangle\phi_{0},\quad K_{\theta}\phi_{j,l} & =\sum_{m}\langle g_{\theta},\phi_{m,1}\rangle(\phi_{j,l}\ast\phi_{m,1})=\frac{1}{\sqrt{2}}\langle g_{\theta},\phi_{j,1}\rangle\phi_{j,l},\quad j\in\N,l\in\{0,1\}
\end{align*}
 by the angle sum identities (for non-symmetric kernels $K_{\theta}$
could be diagonolised by the complex valued Fourier basis). We thus
obtain the singular value decomposition $K_{\theta}f=\sum_{k}\rho_{\theta,k}f_{k}\phi_{k}$,
again in muli-index notation $k=(j,l),j\in\N,l\in\{0,1\}$, where
$\rho_{\theta,k}=\langle g_{\theta},\phi_{|k|,1}\rangle/\sqrt{2}$
and $f_{k}=\langle f,\phi_{k}\rangle$. Depending on the regularity
of $g$ and thus the decay of $\langle g_{\theta},\phi_{j,1}\rangle$
the problem is mildly or severely ill-posed.

If the convolution kernel is fully unknown, we parametrise $g_{\theta}=\theta$ by all (symmetric) 1-periodic kernels $\theta$. Due to the SVD, we then can identify $g_\theta$ with the singular values, that is, we set $\theta=(\rho_{\theta,k})_k$. The sample $T$ can be understood as training data, where the convolution experiment is applied to all basis functions $f\in\{\phi_{j,l}\}$. In this scenario we obtain $\epsilon=\delta$. 

In our simulation $\theta_0$ is given by the periodic Laplace kernel 
$g_{\theta_0}(x)=\frac{1}{C_h}e^{-|x|/h}\1_{[-1/2,1/2]}(x)$
with normalisation constant $C_h=2h(1+e^{-1/(2h)})$ and fixed bandwidth $h=0.1$. Hence, we have for $k\in\N\times\{0,1\}$
\[
\rho_{\theta_0,0}=1,\qquad\rho_{\theta_0,k}=\frac{2h^{-1}}{C_{h}(4\pi^{2}|k|^{2}+h^{-2})}\big(1-e^{-1/(2h)}\cos(\pi|k|)+e^{-1/(2h)}2\pi|k|h\sin(\pi|k|)\big)
\]
In particular, we have two degree of illposedness. We moreover use $f_{0}$ from (\ref{eq:f0}). 

To implement the empirical Bayes procedure with the trigonometric
basis and corresponding approximation spaces $V_{j}=\Span(\phi_{k}:|k|\le j)$,
we need to replace $2^{j}$ by $2j$ as mentioned in Remark~\ref{rem:FourierBasis}.
Choosing some $b>1$ and setting $\mathcal{J}_{\epsilon}=\Big\lfloor\frac{\log\epsilon^{-1}}{(s_{0}+t+d/2)\log b}\Big\rfloor$
for some lower bound $s_{0}\le s$, the selection rule then reads
as
\[
\hat{J}:=\min\big\{ j\in\{1,b,b^{2},\dots,b^{\mathcal{J}_{\epsilon}}\}:\|\hat{f}_{i}-\hat{f}_{j}\|\le\Delta\epsilon(\log\epsilon^{-1})^{2}i^{3/2}\forall i>j\big\}.
\]

\begin{figure}

\includegraphics[width=0.5\textwidth]{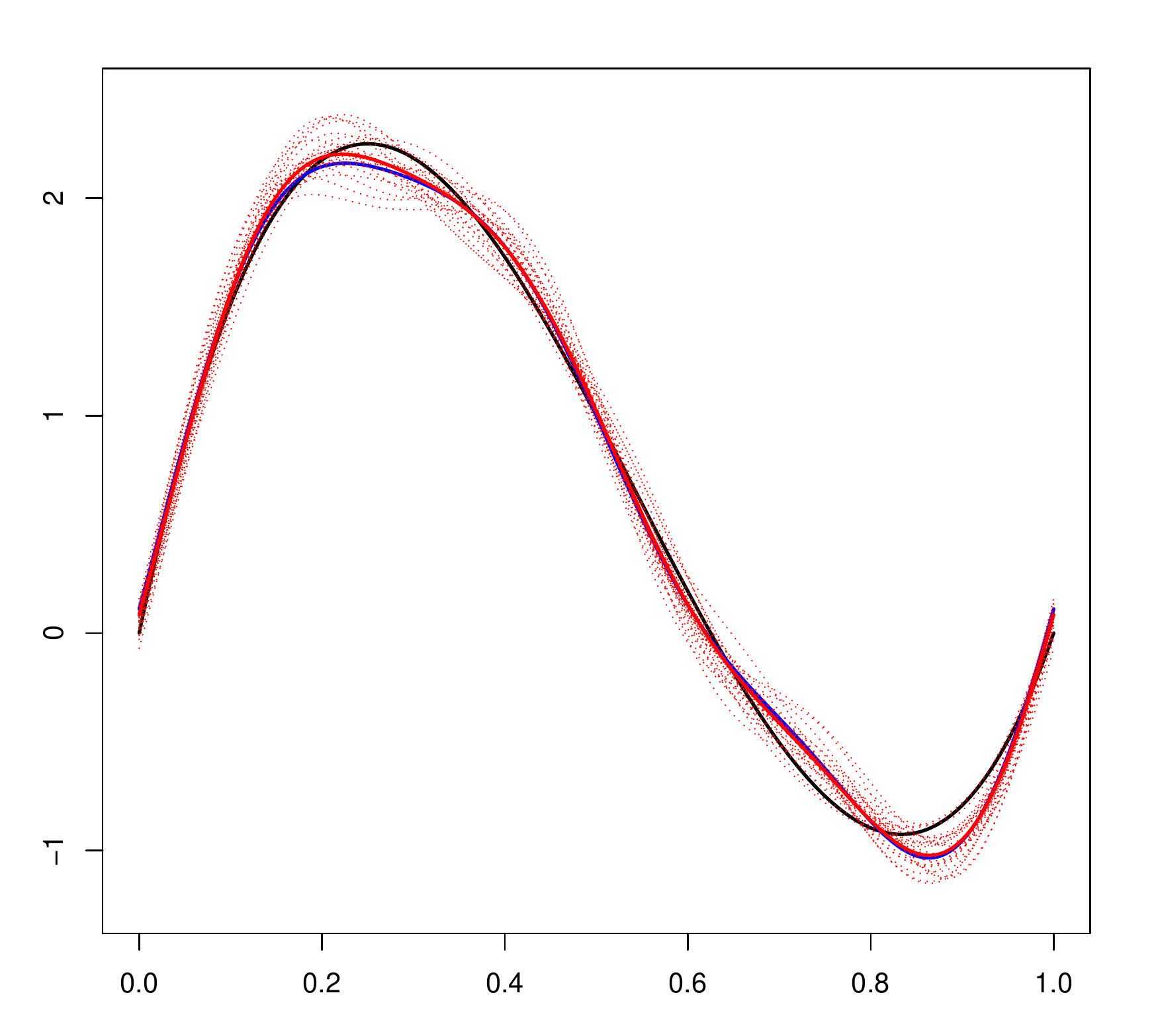}\includegraphics[width=0.5\textwidth]{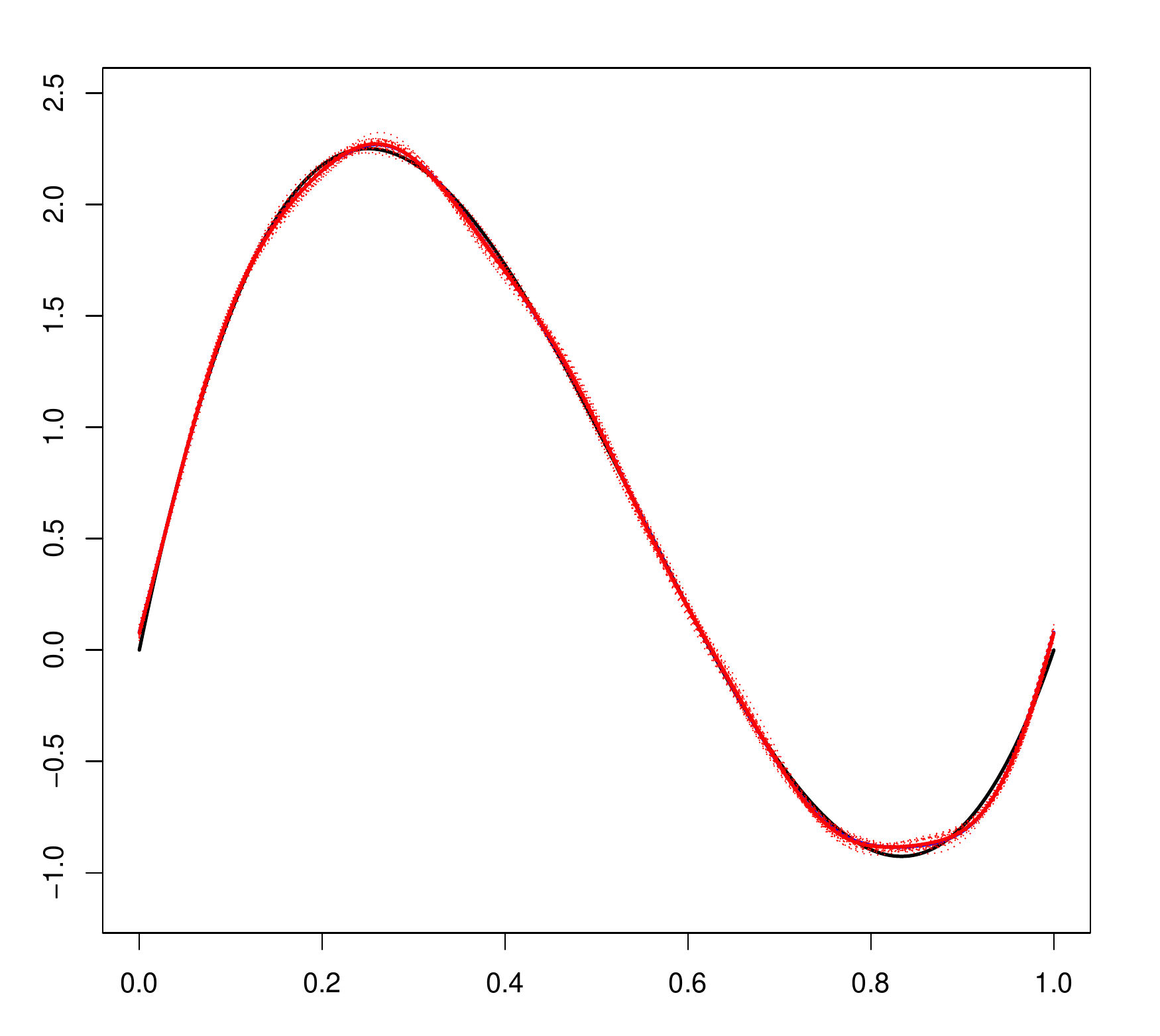}\caption{\label{fig:Deconvolution}Adaptive deconvolution for the Laplace kernel with $\epsilon=\delta=10^{-2}$ (left) and $\epsilon=\delta=10^{-3}$ (right): True function (black), projection estimator (blue), posterior mean (red, solid) and 20 draws from the posterior distribution (red, dotted).}

\end{figure}

Using the again Gaussian product priors for $f$ and $\theta$, the posterior distribution can be similarly approximated as described in Section~\ref{sec:Heat}. However, the nuisance parameter $\theta$ is now infinite dimensional. Here, we can profit from the truncated product structure of the prior which implies that the posterior distribution only depends on the $\hat J$-dimensional projection $P_{\hat J}\theta$ (note that Assumption~\ref{ass:bases} is satisfied with $m=0$). More precisely, we only have to draw from the posterior given by 
\begin{equation*}
\Pi(B|Y,T)=\frac1C\int_{B}\exp\Big(\frac1{\epsilon^2}\langle P_{\hat J}K_{\theta}f,Y\rangle-\frac1{2\epsilon^{2}}\|P_{\hat J}K_{\theta}f\|^{2}+\frac 1{\delta^2}\langle P_{\hat J}\theta,T\rangle-\frac1{2\delta^{2}}\|P_{\hat J}\theta\|^{2}\Big)\d\Pi(f,\theta),
\end{equation*}
with normalisation constant $C>0$ and for all Borel sets $B\subset L^2\times P_{\hat J}\Theta$, cf. proof of Theorem~\ref{thm:contraction}. Therefore, a Gibbs sampler can be used to draw successively the coordinates of $P_{\hat J}\theta$ with a Metropolis-Hastings algorithm and iterate as above with draws of $f$. This simulation approach is not restricted to this particular example, but applies generally. Note that in the specific deconvolution setting, the map $\theta\mapsto K_\theta f$ is linear for fixed $f$, such that $\theta|f,Y,T$ can be directly sampled from a Gaussian distribution. 

For $\epsilon=\delta=10^{-2}$ and $\epsilon=\delta=10^{-3}$ a typical trajectory of the posterior
mean and 20 draws from the posterior are presented in Figure~\ref{fig:Deconvolution}
where the Lepski rule has chosen $J=3$ (i.e. 7 basis functions) and $J=5$ (11 basis functions), respectively.
For the larger noise level, the posterior mean slightly improves the Galerkin projector, while for the smaller noise level both estimators basically coincide.
We see a much better concentration of the posterior distribution than
in the severely ill-posed case discussed previously. In a Monte Carlo simulation for $\epsilon=\delta=10^{-2}$
based on 500 iterations in this setting the posterior mean for $f$
achieved a RMISE of 0.1142 which is approximately $8.6\%$ of $\|f_{0}\|$.
The Lepski method has chosen $J\in\{2,3\}$ with relative frequency
$0.97$. For $\epsilon=\delta=10^{-2}$ the simulation yields a RMISE of 0.0174, which is $1.3\%$ of $\|f_0\|$, and projections levels $J$ in $\{4,5\}$ in $0.82\%$ of the Monte Carlo iterations.

\section{Proofs\label{sec:Proofs}}
We first study some smoothing properties of the operator $K_\theta$.
\begin{lem}
\label{lem:galerkin}Under Assumption~\ref{ass:ill-posed} we have
$\|K_{\theta,j}^{-1}\|_{V_{j}\to V_{j}}\le Q\sigma_{j}^{-1}$ for
all $\theta\in\Theta$.
\end{lem}

\begin{proof}
For $g\in V_{j}$ the function $h=K_{\theta,j}^{-1}g\in V_{j}$ is
given by the unique solution to the linear system
\[
\langle K_{\theta}h,v\rangle=\langle g,v\rangle,\qquad\text{for all }v\in V_{j}.
\]
Assumption~\ref{ass:ill-posed} then yields
\begin{align*}
\sigma_{j}\|h\|^{2} & =\sigma_{j}\sum_{|k|\le j}\langle h,\phi_{k}\rangle^{2}\le\sum_{|k|\le j}\sigma_{|k|}\langle h,\phi_{k}\rangle^{2}\\
 & =Q\langle K_\theta h,h\rangle\le Q\|h\|\sup_{v\in V_{j}:\|v\|\le1}\langle K_{\theta}h,v\rangle=Q\|h\|\sup_{v\in V_{j}:\|v\|\le1}\langle g,v\rangle=Q\|h\|\|g\|.
\end{align*}
Therefore, $\sigma_{j}\|K_{\theta,j}^{-1}g\|\le Q\|g\|$ holds true
for all $g\in V_{j}$.
\end{proof}

\begin{rem}\label{rem:Kf2}
  As soon as $(\sigma_j)$ decays at least geometrically, Assumptions~\ref{ass:bases} and \ref{ass:ill-posed} also yield $\|K_\theta f\|^2\lesssim \sum_{k}\sigma_{|k|}^2\langle f,\phi_k\rangle^2$. Indeed, we have for any $f\in L^2$ such that the right-hand side is finite:
  \begin{align*}
    \|K_\theta f\|^2&=\sum_k|\langle K_\theta P_{|k|+m}f,\phi_k\rangle|^2\le\sum_k\|K_\theta^{1/2}P_{|k|+m}f\|^2\|K_\theta^{1/2}\phi_k\|^2\\
    &\lesssim\sum_k\sum_{|l|\le |k|+m}\sigma_{|k|}\sigma_{|l|}\langle f,\phi_l\rangle^2=\sum_l\Big(\sum_{|k|\ge(|l|-m)\vee0}\sigma_{|k|}\Big)\sigma_{|l|}\langle f,\phi_l\rangle^2\lesssim\sum_l\sigma^2_{|l|}\langle f,\phi_l\rangle^2.
  \end{align*}
\end{rem}

\subsection{Proof of Proposition~\ref{prop:concHatF}}

To simplify the notation, we abbreviate $\P=\P_{f,\theta}$ in the
sequel and define the operator $\Delta_{T,j}:=K_{T,j}-K_{\theta,j}$.
Set for $\gamma\in(0,1-Q/\tau)$
\[
\Omega_{T,j}:=\{\|K_{\theta,j}^{-1}\Delta_{T,j}\|_{V_{j}\to V_{j}}\le\gamma\}.
\]
Lemma~\ref{lem:galerkin} yields
\[
\P\big(\Omega_{T,j}^{c}\big)\le\P\big(\|K_{\theta,j}^{-1}\|_{V_{j}\to V_{j}}\|\Delta_{T,j}\|_{V_{j}\to V_{j}}>\gamma\big)\le\P\big(\|\Delta_{T,j}\|_{V_{j}\to V_{j}}>\gamma\sigma_{j}/Q\big).
\]
Under Assumption~\ref{ass:operator} we have due to the condition
$\delta\sigma_{j}^{-1}(\kappa+\sqrt{d_{j}})\le\gamma/(CQL)$ 
\begin{align*}
\P\big(\Omega_{T,j}^{c}\big)\le & \P\big(\|\Delta_{T,j}\|_{V_{j}\to V_{j}}>\gamma\sigma_{j}/Q\big)\\
\le & \P\Big(\|P_{j}W\|_{j}\ge\frac{\gamma\sigma_{j}}{QL}\Big)\le\P\big(\|P_{j}W\|_{V_{j}\to V_{j}}>C\delta(\kappa+\sqrt{d_{j}})\big)\le e^{-c\kappa^{2}}.
\end{align*}
We thus may restrict on $\Omega_{T,j}$ on which the operator $K_{T,j}=K_{\theta,j}(\Id-K_{\theta,j}^{-1}\Delta_{T,j})$
is invertible satisfying
\[
\|K_{T,j}^{-1}\|_{V_{j}\to V_{j}}\le\|(\Id-K_{\theta,j}^{-1}\Delta_{T,j})^{-1}\|_{V_{j}\to V_{j}}\|K_{\theta,j}^{-1}\|_{V_{j}\to V_{j}}\le\frac{1}{1-\gamma}\|K_{\theta,j}^{-1}\|_{V_{j}\to V_{j}}\le\frac{Q}{(1-\gamma)\sigma_{j}}
\]
where we used Lemma~\ref{lem:galerkin} in the last step. Hence,
for $\gamma\le1-Q/\tau$ we have $\Omega_{T,j}\subset\{\|K_{T,j}^{-1}\|_{V_{j}\to V_{j}}\le\tau\sigma_{j}^{-1}\}$.
Therefore, we can decompose on $\Omega_{T,j}$ 
\begin{align}
\|\hat{f}_{j}-f\|^{2} & =\|P_{j}f-f\|^{2}+\|\hat{f}_{j}-P_{j}f\|^{2}\nonumber \\
 & \le\|P_{j}f-f\|^{2}+\|K_{T,j}^{-1}P_{j}Y-P_{j}f\|^{2}.\label{eq:decomp}
\end{align}
The first term is the usual bias. For the second term in (\ref{eq:decomp})
we write on $\Omega_{T,j}$ 
\begin{align*}
K_{T,j}^{-1}P_{j}Y-P_{j}f & =\big((\Id-K_{\theta,j}^{-1}\Delta_{T,j})^{-1}-\Id\big)P_{j}f+\epsilon(\Id-K_{\theta,j}^{-1}\Delta_{T,j})^{-1}K_{\theta,j}^{-1}P_{j}Z\\
 & =(\Id-K_{\theta,j}^{-1}\Delta_{T,j})^{-1}K_{\theta,j}^{-1}\Delta_{T,j}P_{j}f+\epsilon(\Id-K_{\theta,j}^{-1}\Delta_{T,j})^{-1}K_{\theta,j}^{-1}P_{j}Z.
\end{align*}
Since $\|(\Id-K_{\theta,j}^{-1}\Delta_{T,j})^{-1}\|_{V_{j}\to V_{j}}\le1/(1-\gamma)$
on $\Omega_{T,j}$, we obtain
\begin{align}
\|K_{T,j}^{-1}P_{j}Y-P_{j}f\| & \le\frac{1}{1-\gamma}\|K_{\theta,j}^{-1}\|_{V_{j}\to V_{j}}\|\Delta_{T,j}\|_{V_{j}\to V_{j}}\|P_{j}f\|+\frac{\epsilon}{1-\gamma}\|K_{\theta,j}^{-1}\|_{V_{j}\to V_{j}}\|P_{j}Z\|\nonumber \\
 & \le\frac{Q}{(1-\gamma)\sigma_{j}}\big(\|f\|\|\Delta_{T,j}\|_{V_{j}\to V_{j}}+\epsilon\|P_{j}Z\|\big).\label{eq:bound}
\end{align}
To deduce a concentration inequality for $\|P_{j}Z\|$, we proceed
as proposed in \citep{gineNickl2011}: For a countable dense subset
$B$ of the unit ball in $L^{2}$, we have $\|P_{j}Z\|=\sup_{f\in B}\|P_{j}Z(f)\|$.
The Borell-Sudakov-Tsirelson inequality \citep[Thm. 2.5.8]{gineNickl2016}
yields for any $\kappa>0$
\begin{align*}
\P(\|P_{j}Z\|\ge\kappa+\E[\|P_{j}Z\|]) & \le\P\Big(\sup_{f\in B}\|P_{j}Z(f)\|-\E\Big[\sup_{f\in B}\|P_{j}Z(f)\|\Big]\ge\kappa\Big)\le2^{-\kappa^{2}/(2\sigma^{2})}
\end{align*}
with $\sigma^{2}=\sup_{f\in B}\Var(P_{j}Z(f))\le\|f\|^2\le1$. Since
\[
\E[\|P_{j}Z\|]\le\E[\|P_{j}Z\|^{2}]^{1/2}=\Big(\sum_{|k|\le j}\E[Z_{k}^{2}]\Big)^{1/2}=d_{j}^{1/2}
\]
and $d_{j}\lesssim\epsilon^{-2}\kappa^{2}$, we find for some constant
$C>0$
\[
\P\big(\epsilon\|P_{j}Z\|\ge C\kappa\big)\le\P\big(\epsilon\|P_{j}Z\|\ge\kappa+\epsilon d_{j}^{1/2}\big)\le e^{-\kappa^{2}/(2\epsilon^{2})}.
\]
Under Assumption~\ref{ass:operator} and due to $d_{j}\lesssim\delta^{-2}\kappa^{2}$,
we analogously obtain
\[
\P\big(\|\Delta_{T,j}\|_{V_{j}\to V_{j}}\ge C\kappa\big)\le e^{-\kappa^{2}/(2\delta^{2})}.
\]
In combination with (\ref{eq:bound}), the asserted concentration
inequality is proven.\qed

\subsection{Proof of Theorem~\ref{thm:contraction}}

We proof the theorem in two steps.

\emph{Step 1:} We construct tests $\Psi_{n}=\Psi_{n}(Y,T)$ such that
\begin{align}
\E_{f_{0},\theta_{0}}[\Psi_{n}]\to0\qquad\text{and}
\sup_{f\in\mathcal{F}_{n},\theta\in\Theta:\|f-f_{0}\|\ge D\xi_{n}}\E_{f,\theta}[1-\Psi_{n}]\le 3e^{-(C_{1}+4)\kappa_{n}^{2}/(\epsilon_{n}\vee\delta_{n})^{2}}.\label{eq:test}
\end{align}
Based on the estimator $\hat{f}_{j_n}$ from (\ref{eq:estimator}),
we set 
\[
\Psi_{n}:=\1_{\{\|\hat{f}_{j_n}-f_{0}\|\ge D_{1}\xi_{n}\}}
\]
for $D_{1}=2Cc_{2}\sqrt{C_{1}+4}R+2C_{0}\xi_{n}$ with
the constant $C$ from Proposition~\ref{prop:concHatF}. Due to Proposition~\ref{prop:concHatF}
and $\kappa_{n}/\sigma_{j_{n}}\le c_{2}\xi_{n}$, we then have
\begin{align*}
\E_{f_{0},\theta_{0}}[\Psi_{n}] & =\P_{f_{0},\theta_{0}}\Big(\|\hat{f}_{j_n}-f_{0}\|\ge2C\sigma_{j_{n}}^{-1}\sqrt{C_{1}+4}\kappa_{n}\|f_{0}\|+2C_{0}\xi_{n}\Big)\\
 & \le\P_{f_{0},\theta_{0}}\big(\|\hat{f}_{j_n}-f_{0}\|\ge C\sigma_{j_{n}}^{-1}\sqrt{C_{1}+4}\kappa_{n}\|f_{0}\|+\|f_{0}-P_{j_{n}}f_{0}\|\big)\le 3e^{-(C_{1}+4)\kappa_{n}^{2}/(\epsilon_{n}\vee\delta_{n})^{2}}
\end{align*}
converging to 0.

On the alternative we set $D=D_{2}(1+R)$ for $D_{2}=2\max(C_{0}+D_{1},C\sqrt{C_{1}+4}/c_{2})$.
For any $\theta\in\Theta$ and any $f\in\mathcal{F}_{n}$ with $\|f-f_{0}\|\ge D_{2}(1+R)\xi_{n}$
we have $(2-D_{2}\xi_{n})\|f-f_{0}\|\ge D_{2}(1+R)\xi_{n}$ for sufficiently
small $\xi_{n}\downarrow0$. Therefore,
\begin{align}
\|f-f_{0}\|\ge & \frac{D_{2}}{2}(1+R+\|f-f_{0}\|)\xi_{n}\ge\frac{D_{2}}{2}\big(1+\|f_{0}\|+\|f-f_{0}\|\big)\xi_{n}\nonumber \\
\ge & \frac{D_{2}}{2}(1+\|f\|)\xi_{n}\ge C\sigma_{j_{n}}^{-1}\sqrt{C_{1}+4}\kappa_{n}\|f\|+(C_{0}+D_{1})\xi_{n},\label{eq:alternative}
\end{align}
where the last inequality holds by the choice of $D_{2}$.
We obtain
\begin{align*}
\E_{f,\theta}[1-\Psi_{n}] & =\P_{f,\theta}\big(\|\hat{f}_{j_n}-f_{0}\|<D_{1}\xi_{n}\big)\\
 & \le\P_{f,\theta}\big(\|\hat{f}_{j_n}-f\|>\|f-f_{0}\|-D_{1}\xi_{n}\big)\\
 & \le\P_{f,\theta}\big(\|\hat{f}_{j_n}-f\|>C\sigma_{j_{n}}^{-1}\sqrt{C_{1}+4}\kappa_{n}\|f\|+C_{0}\xi_{n}\big).
\end{align*}
Proposition~\ref{prop:concHatF} yields again $\E_{f,\theta}[1-\Psi_{n}]\le 3e^{-(C_{1}+4)\kappa_{n}^{2}/(\epsilon_{n}\vee\delta_{n})^{2}}$.

\emph{Step 2:} Since $\E_{f_{0},\theta_{0}}[\Psi_{n}]\to0$, it suffices
to prove that 
\begin{align*}
\Pi_{n}\big(f & \in V_{j_n}:\|f-f_{0}\|>D\xi_{n}|Y,T\big)(1-\Psi_{n})\\
= & \frac{\int_{f\in V_{j_n}:\|f-f_{0}\|>D\xi_{n},\theta\in\Theta}p_{f,\theta}(Z,W)\d\Pi_{n}(f,\theta)(1-\Psi_{n})}{\int_{f\in\mathcal{F},\theta\in\Theta}p_{f,\theta}(Z,W)\d\Pi_{n}(f,\theta)}\to0\quad\text{in }P_{f_{0},\theta_{0}}\text{-probability.}
\end{align*}
Due to Assumption~\ref{ass:bases}, we have $K_{\theta}P_{j_n}=P_{j_n+m}K_{\theta}P_{j_n}=K_{\theta,j_n+m}P_{j_n}$.
Hence, restricted on $f\in V_{j_n}$, we obtain
\begin{align*}
p_{f,\theta}(z,w) & =\exp\Big(\frac{1}{\epsilon}\langle K_{\theta,j_n+m}f-K_{\theta_{0}}f_{0},z\rangle-\frac{1}{2\epsilon^{2}}\|K_{\theta,j_n+m}f-K_{\theta_{0}}f_{0}\|^{2}\\
 & \hspace{5em}+\frac{1}{\delta}\langle\theta-\theta_{0},w\rangle-\frac{1}{2\delta^{2}}\|\theta-\theta_{0}\|^{2}\Big)\\
 & =\exp\Big(\frac{1}{\epsilon}\langle K_{\theta,j_n+m}f-K_{\theta_{0}}f_{0},z\rangle-\frac{1}{2\epsilon^{2}}\|K_{\theta,j_n+m}f-P_{j_n+m}K_{\theta_{0}}f_{0}\|^{2}\\
 & \hspace{5em}-\frac{1}{2\epsilon^{2}}\|(\Id-P_{j_n+m})K_{\theta_{0}}f_{0}\|^{2}+\frac{1}{\delta}\langle\theta-\theta_{0},w\rangle-\frac{1}{2\delta^{2}}\|\theta-\theta_{0}\|^{2}\Big).
\end{align*}
Since we assume that $K_{\theta,j_n+m}$ depends only on $P_{j_n+m}\theta=(\theta_{1},\dots,\theta_{l_{j_n+m}})$
and $\Pi$ is a product prior in $(\theta_{k})$, we may rewrite
\begin{align}
\Pi_{n}\big(f & \in V_{j_n}:\|f-f_{0}\|>D\xi_{n}|Y,T\big)(1-\Psi_{n})\label{eq:ratio}\\
 & \qquad\le\frac{\int_{f\in\mathcal{F}\cap V_{j_n}:\|f-f_{0}\|>D\xi_{n},\theta\in\Theta}p_{f,\theta}^{(j_n)}(Z,W)\d\Pi_{n}(f,\theta)(1-\Psi_{n})}{\int_{f\in\mathcal{F}\cap V_{j_n},\theta\in\Theta}p_{f,\theta}^{(j_n)}(Z,W)\d\Pi_{n}(f,\theta)}\nonumber 
\end{align}
with
\begin{align*}
p_{f,\theta}^{(j_n)}(z,w) & =\exp\Big(\frac{1}{\epsilon}\langle P_{j_n+m}(K_{\theta}f-K_{\theta_{0}}f_{0}),z\rangle-\frac{1}{2\epsilon^{2}}\|P_{j_n+m}(K_{\theta}f-K_{\theta_{0}}f_{0})\|^{2}\\
 & \hspace{5em}+\frac{1}{\delta}\langle P_{j_n+m}(\theta-\theta_{0}),w\rangle-\frac{1}{2\delta^{2}}\|P_{j_n+m}(\theta-\theta_{0})\|^{2}\Big).
\end{align*}

We can proceed as in the proof of the Theorems 7.3.1 and 7.3.5, respectively, in
\citep{gineNickl2016}. First we need a lower bound for the denominator in \eqref{eq:ratio}. Defining the event 
\[
  B_n:=\Big\{(f,\theta)\in V_{j_{n}}\times\Theta:\frac{\|P_{j_{n}+m}(K_{\theta}f-K_{\theta_{0}}f_{0})\|^{2}}{\epsilon_{n}^{2}}+\frac{\|P_{j_{n}+m}(\theta-\theta_{0})\|^{2}}{\delta_{n}^{2}}\le\frac{\kappa_{n}^{2}}{(\epsilon_{n}\vee\delta_{n})^{2}}\Big\},
\]
we obtain
\begin{align*}
  &\P_{f_0,\theta_0}\Big(\int_{f\in\mathcal{F}\cap V_{j_n},\theta\in\Theta}p_{f,\theta}^{(j_n)}(Z,W)\d\Pi_{n}(f,\theta)\ge e^{-(C_1+2)\kappa_n^2/(\epsilon_n\vee\delta_n)^2}\Big)\\
  &\qquad\ge \P_{f_0,\theta_0}\Big(\int_{B_n}p_{f,\theta}^{(j_n)}(Z,W)\frac{\d\Pi_{n}(f,\theta)}{\Pi_n(B_n)}\ge e^{-2\kappa_n^2/(\epsilon_n\vee\delta_n)^2}\Big)
  \ge 1-\frac{(\epsilon_n\vee\delta_n)^{2}}{\kappa_n^2},
\end{align*}
where the first inequality is due to the small ball probability (\ref{eq:smallBall}) and the second inequality follows along the lines of Lemma 7.3.4 in \citep{gineNickl2016}. Using this bound for the denominator together with Markov's inequality and Fubini's theorem, the probability that \eqref{eq:ratio} is larger than some $r>0$ is bounded by
\begin{align*}
  &\P_{f_{0},\theta_{0}}\Big(\Pi_{n}(f\in V_{j_n}:\|f-f_{0}\|>D\xi_{n}|Y,T)(1-\Psi_n)\ge r\Big)\\
  &\qquad\le \P_{f_{0},\theta_{0}}\Big(e^{(C_1+2)\kappa_n^2/(\epsilon_n\vee\delta_n)^2}(1-\Psi_n)\int_{f,\theta:\|f-f_{0}\|>D\xi_{n}}p_{f,\theta}^{(j_n)}(Z,W)\d\Pi_{n}(f,\theta)\ge r\Big)+\frac{(\epsilon_n\vee\delta_n)^{2}}{\kappa_n^2}\\
  &\qquad\le \frac{e^{(C_1+2)\kappa_n^2/(\epsilon_n\vee\delta_n)^2}}r\E_{f_0,\theta_0}\Big[(1-\Psi_n)\int_{f,\theta:\|f-f_{0}\|>D\xi_{n}}p_{f,\theta}^{(j_n)}(Z,W)\d\Pi_{n}(f,\theta)\Big]+\frac{(\epsilon_n\vee\delta_n)^{2}}{\kappa_n^2}\\
  &\qquad\le \frac{e^{(C_1+2)\kappa_n^2/(\epsilon_n\vee\delta_n)^2}}r\int_{f,\theta:\|f-f_{0}\|>D\xi_{n}}\E_{f_0,\theta_0}\Big[(1-\Psi_n)p_{f,\theta}^{(j_n)}(Z,W)\Big]\d\Pi_{n}(f,\theta)+\frac{(\epsilon_n\vee\delta_n)^{2}}{\kappa_n^2}.
\end{align*}
Note that $p_{f,\theta}^{(j_n)}$ corresponds to the density of the law of $(Y',T')$ where 
$$Y'=P_{j_n+m}K_{\theta}f+(\Id-P_{j_n+m})K_{\theta_{0}}f_{0}+\epsilon Z\quad\text{and}\quad T'=P_{j_n+m}\theta+(\Id-P_{j_n+m})\theta_{0}+\delta W$$ with respect to $\P_{\theta_{0},f_{0}}^{Y}\otimes\P_{\theta_{0}}^{T}$ and we have $\Psi_{n}(Y,T)=\Psi_{n}(Y',T')$ by construction. Therefore, we can apply Step~1 to bound the previous display and conclude
\begin{equation}
\P_{f_{0},\theta_{0}}\Big(\Pi_{n}(f\in V_{j_n}:\|f-f_{0}\|>D\xi_{n}|Y,T)\ge r\Big)\lesssim\frac{1}{r}e^{-2\kappa_n^2/(\epsilon_n\vee\delta_n)^2}+\frac{(\epsilon_n\vee\delta_n)^{2}}{\kappa_n^2}+\E_{f_0,\theta_0}[\Psi_n].\label{eq:boundRate}
\end{equation}
It remains to note that for any $r>0$ the right-hand side converges to zero as $n\to\infty$ .  \qed

\subsection{Proof of Theorem~\ref{thm:mildly}}

For the sake of brevity we omit the subscript $n$ in the proof. $c_{1},c_{2},\dots$
will denote positive, universal constants. We will choose $\kappa,\xi$
and $j=J$ according to
\begin{equation}
\xi\simeq\Big((\epsilon\vee\delta)\log\frac{1}{\epsilon\vee\delta}\Big)^{2s/(2s+2t+d)},\kappa\simeq\Big((\epsilon\vee\delta)\log\frac{1}{\epsilon\vee\delta}\Big)^{2(s+t)/(2s+2t+d)},2^{j}=\kappa^{-1/(s+t)}.\label{eq:choices}
\end{equation}
It is not difficult to see that these choices satisfy the requirements
of Theorem~\ref{thm:contraction} and $\|f_{0}-P_{j}f_{0}\|\lesssim\xi$
holds by (\ref{eq:approx}). Moreover, the support of $\Pi_{f}$ lies
in $V_{j}$ such that (\ref{eq:smallBias}) is trivially satisfied
for $\F_{n}=\{f:\|f-P_{j}f\|\le C_{0}\xi\}$. It only remains to verify
the small ball probability (\ref{eq:smallBall}). 

Owing to $P_{j}K_{\theta}=P_{j}K_{\theta}P_{j+m}=P_{j}K_{\theta,j+m}$,
(\ref{eq:approx}) and $\|K_{\theta}f\|\lesssim\|f\|_{H^{-t}}$, we
can estimate for any $f\in V_{j}$
\begin{align*}
\|P_{j+m}(K_{\theta}f-K_{\theta_{0}}f_{0})\|\le & \|P_{j+m}(K_{\theta}-K_{\theta_{0}})f_{0}\|+\|P_{j+m}K_{\theta,j+2m}(f-f_{0})\|\\
\le & \|(K_{\theta,j+2m}-K_{\theta_{0},j+2m})f_{0}\|+\|K_{\theta}P_{j+2m}(f-f_{0})\|\\
\le & R\|K_{\theta,j+2m}-K_{\theta_{0},j+2m}\|_{V_{j+2m}\to V_{j+2m}}+\|f-P_{j+2m}f_{0}\|_{H^{-t}}\\
\le & R\|K_{\theta,j+2m}-K_{\theta_{0},j+2m}\|_{V_{j+2m}\to V_{j+2m}}+\|f-P_{j}f_{0}\|_{H^{-t}}\\
 & \quad+\|P_{j+2m}(\Id-P_{j})f_{0}\|_{H^{-t}}.
\end{align*}
The last term is bounded by $\|(\Id-P_{j})f_{0}\|_{H^{-t}}\lesssim2^{-j(s+t)}\|f_{0}\|_{H^{s}}$
being of the order $\mathcal{O}(\kappa\epsilon/(\epsilon\vee\delta))$
due to $\epsilon\gtrsim\delta$ and the choice $j$ as in (\ref{eq:choices}).
We obtain 
\begin{align}
 & \Pi\big(f\in\mathcal{F}\cap V_{j},\theta\in\Theta:\epsilon^{-2}\|P_{j+m}(K_{\theta}f-K_{\theta_{0}}f_{0})\|^{2}+\delta^{-2}\|P_{j+m}(\theta-\theta_{0})\|^{2}\le\kappa^{2}/(\epsilon\vee\delta)^{2}\big)\nonumber \\
 & \quad\ge\Pi\Big(f\in\mathcal{F}\cap V_{j},\theta\in\Theta:\frac{1}{\epsilon^{2}}\|f-P_{j}f_{0}\|_{H^{-t}}^{2}+\frac{R^{2}}{\epsilon^{2}}\|K_{\theta,j+2m}-K_{\theta_{0},j+2m}\|_{V_{j+2m}\to V_{j+2m}}^{2}\nonumber \\
 & \quad\qquad\qquad+\frac{1}{\delta^{2}}\|P_{j+m}(\theta-\theta_{0})\|^{2}\le c_{1}\kappa^{2}/(\epsilon\vee\delta)^{2}\Big)\nonumber \\
 & \quad\ge\Pi_{f}\Big(f\in\mathcal{F}\cap V_{j}:\|f-P_{j}f_{0}\|_{H^{-t}}\le\frac{c_{1}\epsilon\kappa}{\sqrt{2}(\epsilon\vee\delta)}\Big)\nonumber \\
 & \quad\qquad\times\Pi_{\theta}\Big(\frac{1}{\epsilon^{2}}\|K_{\theta,j+2m}-K_{\theta_{0},j+2m}\|_{V_{j+2m}\to V_{j+2m}}^{2}+\frac{1}{\delta^{2}}\|P_{j+m}(\theta-\theta_{0})\|^{2}\le\frac{c_{1}\kappa^{2}}{2(\epsilon\vee\delta)^{2}}\Big),\label{eq:inProofRates}
\end{align}
where the last line follows from independence of $f$ and $\theta$
under $\Pi$. The first term can be bounded using the product structure
and the estimate (\ref{eq:raysBound}). Setting $\tilde{\kappa}=\frac{\epsilon\kappa}{\epsilon\vee\delta}$
and taking $\log\tau_{j}\lesssim j$ into account, we obtain
\begin{align*}
 & \Pi_{f}\Big(f\in\mathcal{F}\cap V_{j}:\|f-P_{j}f_{0}\|_{H^{-t}}\le\frac{c_{1}\epsilon\kappa}{\sqrt{2}(\epsilon\vee\delta)}\Big)\\
 & \quad=\Pi_{f}\Big(f\in\mathcal{F}\cap V_{j}:\sum_{|k|\le j}2^{-2t|k|}(f_{k}-f_{0,k})^{2}\le\frac{c_{1}^2\tilde{\kappa}^{2}}{2}\Big)\\
 & \quad\ge\prod_{|k|\le j}\Pi_{f}\big(|f_{k}-f_{0,k}|\le c_{2}\tilde{\kappa}2^{(t-d/2)|k|}\big)\\
 & \quad\ge\exp\Big(c_{3}2^{jd}\log(2\Gamma)+c_{3}\sum_{|k|\le j}\big((2t-d)|k|-\log\tau_{|k|}+\log\tilde{\kappa}\big)-2\gamma\sum_{|k|\le j}\tau_{|k|}^{-2}\big(|f_{0,k}|^{2}+c_{3}\tilde{\kappa}^{2}2^{(2t-d)|k|}\big)\Big)\\
 & \quad\ge\exp\Big(c_{4}2^{jd}(\log\tilde{\kappa}-j)-2\gamma\max_{l\le j}(2^{-2sl}\tau_{l}^{-2})\|f_{0}\|_{H^{s}}^{2}-c_{4}\tilde{\kappa}^{2}\tau_{j}^{-2}2^{2tj}\Big).
\end{align*}
Since $\kappa\simeq2^{-j(s+t)}$, we have $\log\tilde{\kappa}^{-1}\lesssim j+\log\frac{\epsilon\vee\delta}{\epsilon}\lesssim j$.
From the assumptions on $\tau_{j}$ we thus deduce
\begin{equation}
\Pi_{f}\Big(f\in\mathcal{F}\cap V_{j}:\|f-P_{j}f_{0}\|_{H^{-t}}\le\frac{c_{1}\epsilon\kappa}{\sqrt{2}(\epsilon\vee\delta)}\Big)\ge e^{c_{5}2^{jd}(\log\tilde{\kappa}-j)}\ge e^{-c_{6}j2^{jd}}.\label{eq:inProofRates2}
\end{equation}
By the the Lipschitz continuity $\|K_{\theta,j}-K_{\theta_{0},j}\|_{V_{j}\to V_{j}}^{2}\lesssim\|P_{j}(\theta-\theta_{0})\|^{2}$,
the second term in (\ref{eq:inProofRates}) is bounded by
\[
\Pi_{\theta}\Big(\Big(\frac{1}{\epsilon^{2}}+\frac{1}{\delta^{2}}\Big)\|P_{j+2m}(\theta-\theta_{0})\|^{2}\le\frac{c_{7}\kappa^{2}}{(\epsilon\vee\delta)^{2}}\Big)\ge\Pi_{\theta}\Big(\|P_{j+m}(\theta-\theta_{0})\|\le\frac{c_{7}(\epsilon\wedge\delta)\kappa}{\epsilon\vee\delta}\Big).
\]
Due to Assumption~\ref{ass:densities} and using again (\ref{eq:raysBound}),
we can estimate for $\bar{\kappa}=\frac{(\epsilon\wedge\delta)\kappa}{\epsilon\vee\delta}$:
\begin{align*}
 & \Pi_{\theta}\Big(\|P_{j+2m}(\theta-\theta_{0})\|\le c_{7}\bar{\kappa}\Big)\\
 & \quad\ge\prod_{|k|\le l_{j+m}}\Pi_{\theta}\big(|\theta_{k}-\theta_{k,0}|\le c_{7}\bar{\kappa}/\sqrt{l_{j+2m}}\big)\\
 & \quad\ge\exp\Big(c_{8}l_{j+2m}\log(2\Gamma)+c_{8}l_{j+2m}\log\bar{\kappa}-\frac{c_{8}}{2}l_{j+2m}\log l_{j+2m}-\gamma\sum_{|k|\le l_{j+2m}}\big(|\theta_{0,k}|+c_{7}\bar{\kappa}/\sqrt{l_{j+2m}}\big)^{2}\Big)\\
 & \quad\ge\exp\Big(c_{9}l_{j+2m}\log\bar{\kappa}-c_{9}l_{j+2m}\log l_{j+2m}-c_{9}\|\theta_{0}\|^{2}-c_{9}\bar{\kappa}^{2}\Big)\\
 & \quad\ge\exp\Big(-c_{10}l_{j+2m}\big(\log(\bar{\kappa}^{-1})+\log l_{j+2m}\big)\Big),
\end{align*}
where we have used in the last step that $\bar{\kappa}\le\kappa\to0$.
Because $\epsilon^{\eta}\lesssim\delta$ implies $\log\bar{\kappa}^{-1}\lesssim j+\log\frac{\epsilon\vee\delta}{\epsilon\wedge\delta}\lesssim j$,
we find in combination with $\log l_{j+2m}\lesssim j$ that
\begin{equation}
\Pi_{\theta}\Big(\|P_{j+m}(\theta-\theta_{0})\|\le c_{7}\bar{\kappa}\Big)\ge e^{-c_{11}jl_{j+2m}}\gtrsim e^{-c_{12}j2^{jd}}.\label{eq:inProofRates3}
\end{equation}
Therefore, (\ref{eq:smallBall}) follows from $j2^{jd}\le\kappa^{2}(\epsilon\vee\delta)^{-2}$,
which is satisfied due to (\ref{eq:choices}), in combination with
(\ref{eq:inProofRates}), (\ref{eq:inProofRates2}) and (\ref{eq:inProofRates3}).\qed

\subsection{Proof of Theorem~\ref{thm:severely}}

The proof is similar to the previous one. The choices of $\kappa,\xi$
and $j$ given by
\begin{equation}
\xi\simeq\Big(\log\frac{1}{\epsilon\vee\delta}\Big)^{-s/t},\qquad\kappa\simeq(\epsilon\vee\delta)^{1/2},\qquad2^{j}=\Big(-\frac{1}{2r}\log\big(\frac{\kappa\epsilon}{\epsilon\vee\delta}\big)\Big)^{1/t}\label{eq:choices-1}
\end{equation}
satisfy the conditions of Theorem~\ref{thm:contraction}. Especially,
we have $\|f_{0}-P_{j}f_{0}\|\lesssim2^{-js}\|f_{0}\|_{H_{s}}\lesssim\xi$
and $2^{J}=2^{j}(1+o(1))$ because of $\log\epsilon\simeq\log\delta$.
Since $\|K_{\theta}f\|^2\lesssim\sum_{k}e^{-2r2^{|k|t}}f_{k}^{2}$,
we estimate for any $f\in\mathcal{F}\cap V_{j}$
\begin{align}
 & \|P_{j+m}(K_{\theta}f-K_{\theta_{0}}f_{0})\|^{2}\nonumber \\
 & \quad\le2\|(K_{\theta,j+2m}-K_{\theta_{0},j+2m})f_{0}\|^{2}+2\|K_{\theta}P_{j+2m}(f-f_{0})\|^{2}\nonumber \\
 & \quad\lesssim2\|K_{\theta,j+2m}-K_{\theta_{0},j+2m}\|_{V_{j+2m}\to V_{j+2m}}^{2}+\sum_{k}e^{-2r2^{|k|t}}\langle f-P_{j+2m}f_{0},\phi_{k}\rangle^{2}\nonumber \\
 & \quad\le2\|K_{\theta,j+2m}-K_{\theta_{0},j+2m}\|_{V_{j+2m}\to V_{j+2m}}+4\sum_{|k|\le j}e^{-2r2^{|k|t}}|f_{k}-f_{0,k}|^{2}+4\sum_{|k|>j}e^{-2r2^{|k|t}}f_{0,k}^{2}.\label{eq:biasSeverely}
\end{align}
Using
\[
\sum_{|k|>j}e^{-2r2^{|k|t}}f_{0,k}^{2}\le2^{-2js}e^{-2r2^{jt}}\|f\|_{H^{s}}^{2}\lesssim e^{-2r2^{jt}},
\]
together with the choice $j$ from (\ref{eq:choices-1}), the last
term in (\ref{eq:biasSeverely}) is $\mathcal{O}(\kappa\epsilon/(\epsilon\vee\delta))$.
Analogously to (\ref{eq:inProofRates}) we obtain for some $c_{1}>0$
\begin{align}
 & \Pi\Big(f\in\mathcal{F}\cap V_{j},\theta\in\Theta:\frac{1}{\epsilon^{2}}\|P_{j+m}(K_{\theta}f-K_{\theta_{0}}f_{0})\|^{2}+\frac{1}{\delta^{2}}\|P_{j+m}(\theta-\theta_{0})\|^{2}\le\frac{\kappa^{2}}{(\epsilon\vee\delta)^{2}}\Big)\nonumber \\
 & \qquad\ge\Pi_{f}\Big(f\in\mathcal{F}\cap V_{j}:\sum_{|k|\le j}e^{-2r2^{|k|t}}|f_{k}-f_{0,k}|^{2}\le c_{1}\frac{\epsilon\kappa}{(\epsilon\vee\delta)}\Big)\nonumber \\
 & \qquad\qquad\times\Pi_{\theta}\Big(\frac{1}{\epsilon^{2}}\|K_{\theta,j+2m}-K_{\theta_{0},j+2m}\|_{V_{j+2m}\to V_{j+2m}}+\frac{1}{\delta^{2}}\|P_{j+m}(\theta-\theta_{0})\|^{2}\le c_{1}\frac{\kappa^{2}}{(\epsilon\vee\delta)^{2}}\Big).\label{eq:inProofRates-1}
\end{align}
The second factor is the same as in the proof of Theorem~\ref{thm:mildly}.
Taking into account that $\log\delta\simeq\log\epsilon$ and (\ref{eq:choices-1})
imply $-\log\frac{(\epsilon\wedge\delta)\kappa}{\epsilon\vee\delta}=-\log\kappa-\log\frac{(\epsilon\wedge\delta)}{\epsilon\vee\delta}\simeq2^{jt}$,
we find
\begin{align}
 & \Pi_{\theta}\Big(\frac{1}{\epsilon^{2}}\|K_{\theta,j+2m}-K_{\theta_{0},j+2m}\|_{V_{j+2m}\to V_{j+2m}}+\frac{1}{\delta^{2}}\|P_{j+m}(\theta-\theta_{0})\|^{2}\le c_{1}\frac{\kappa^{2}}{(\epsilon\vee\delta)^{2}}\Big)\nonumber \\
 & \qquad\ge\Pi_{\theta}\Big(\|P_{j+2m}(\theta-\theta_{0})\|\le\frac{c_{2}(\epsilon\wedge\delta)\kappa}{\epsilon\vee\delta}\Big)\ge e^{-c2^{jt}l_{j+2m}}.\label{eq:inProofRates1-1}
\end{align}
 Setting $\tilde{\kappa}=\frac{\epsilon\kappa}{\epsilon\vee\delta}$
and applying (\ref{eq:raysBound}), we obtain for the first term 
\begin{align*}
 & \Pi_{f}\Big(f\in\mathcal{F}\cap V_{j}:\sum_{|k|\le j}e^{-2r2^{|k|t}}\langle f-f_{0},\phi_{k}\rangle^{2}\le\frac{c_{1}\epsilon\kappa}{(\epsilon\vee\delta)}\Big)\\
 & \quad\ge\prod_{|k|\le j}\Pi_{f}\big(|f_{k}-f_{0,k}|\le c_{3}\tilde{\kappa}2^{-d|k|/2}e^{r2^{|k|t}}\big)\\
 & \quad\ge\exp\Big(c_{4}2^{jd}\log(2\Gamma)+c_{4}\sum_{|k|\le j}\big(r2^{|k|t}-d|k|-\log\tau_{|k|}+\log\tilde{\kappa}\big)\\
 &\qquad\qquad-2\gamma\sum_{|k|\le j}\tau_{|k|}^{-2}\big(|f_{0,k}|^{2}+c_{3}\tilde{\kappa}^{2}2^{-dj}e^{2r2^{|k|t}}\big)\Big)\\
 & \quad\ge\exp\Big(c_{5}2^{jd}(\log\tilde{\kappa}-2^{jt})-c_{5}\max_{l\le j}(2^{-2sl}\tau_{l}^{-2})\|f_{0}\|_{H^{s}}^{2}-c_{5}\tilde{\kappa}^{2}\tau_{j}^{-2}e^{2r2^{jt}}\Big).
\end{align*}
From the assumptions on $\tau_{j}$ and $-\log\tilde{\kappa}\lesssim2^{jt}$
we thus deduce
\begin{equation}
\Pi_{f}\Big(f\in\mathcal{F}\cap V_{j}:\sum_{|k|\le j}e^{-r2^{|k|t}}\langle f-f_{0},\phi_{k}\rangle^{2}\le\frac{c_{1}\epsilon\kappa}{(\epsilon\vee\delta)}\Big)\ge e^{c_{6}2^{jd}(\log\tilde{\kappa}-2^{jt})}\ge e^{-c_{7}2^{j(d+t)}}.\label{eq:inProofRates2-1}
\end{equation}
Combining (\ref{eq:inProofRates1-1}) and (\ref{eq:inProofRates2-1})
yields
\[
\Pi\Big(f\in\mathcal{F}\cap V_{j},\theta\in\Theta:\frac{1}{\epsilon^{2}}\|P_{j+m}(K_{\theta}f-K_{\theta_{0}}f_{0})\|^{2}+\frac{1}{\delta^{2}}\|P_{j+m}(\theta-\theta_{0})\|^{2}\le\frac{\kappa^{2}}{(\epsilon\vee\delta)^{2}}\Big)\ge e^{-c_{8}2^{jt}(2^{jd}+l_{j+2m})}.
\]
Therefore, (\ref{eq:smallBall}) follows from $2^{j(2d+t)}\simeq\log(\kappa^{-1})^{(2d+t)/t}\le\kappa^{2}(\epsilon\vee\delta)^{-2}$
by the choice of $\kappa$ from (\ref{eq:choices-1}).\qed

\subsection{Proof of Theorem~\ref{thm:adaptive}}

Let us introduce the oracle which balances the bias and the variance
term:
\[
J_{o}:=\min\big\{ j\le\mathcal{J}_{\epsilon}:R2^{-js}\le CR\log(1/\epsilon)\epsilon2^{j(t+d/2)}\big\}
\]
where $C$ is the constant from Proposition~\ref{prop:concHatF}
and $R$ is the radius of the Hölder ball. As $\epsilon\to0$ we see
that 
\[
2^{J_{0}}\simeq\big(\epsilon(\log\epsilon^{-1})\big)^{-2/(2s+2t+d)},
\]
which coincides with the choice of $j$ in the proof of Theorem~\ref{thm:mildly}.
The rest of the proof is divided into three steps.

\emph{Step 1:} We will proof that $\hat{J}\le J_{o}$ with probability
approaching one. We have for sufficiently small $\epsilon$
\begin{align*}
\P_{f_{0},\theta_{0}}(\hat{J}>J_{o}) 
 & =\P_{f_{0},\theta_{0}}\big(\exists i>j\ge J_{o}:\|\hat{f}_{i}-\hat{f}_{j}\|>\Delta\epsilon(\log\epsilon^{-1})^{2}2^{i(t+d/2)}\big)\\
 & \le\sum_{i>j\ge J_{o}}\P_{f_{0},\theta_{0}}\big(\|\hat{f}_{i}-\hat{f}_{j}\|>\Delta\epsilon(\log\epsilon^{-1})^{2}2^{i(t+d/2)}\big)\\
 & \le\sum_{i>j\ge J_{o}}\P_{f_{0},\theta_{0}}\big(\|\hat{f}_{i}-f_{0}\|+\|\hat{f}_{j}-f_{0}\|>\Delta\epsilon(\log\epsilon^{-1})^{2}2^{i(t+d/2)}\big)\\
 & \le2\mathcal{J}_{\epsilon}\sum_{j\ge J_{o}}\P_{f_{0},\theta_{0}}\big(\|\hat{f}_{j}-f_{0}\|>2CR\epsilon(\log\epsilon^{-1})2^{j(t+d/2)}\big).
\end{align*}
By definition of $J_{o}$ we have for every $j\ge J_{o}$ and $f_{0}\in H^{s}(R)$
that $\|f_{0}-P_{j}f_{0}\|\le R2^{-js}\le CR\log(1/\epsilon)\epsilon2^{j(t+d/2)}$.
Hence, for $\epsilon$ sufficiently small we obtain
\[
\P_{f_{0},\theta_{0}}(\hat{J}>J_{o})\le2\mathcal{J}_{\epsilon}\sum_{j\ge J_{0}}\P_{f_{0},\theta_{0}}\big(\|\hat{f}_{j}-f_{0}\|>C\|f_{0}\|\epsilon(\log\epsilon^{-1})2^{j(t+d/2)}+\|f_{0}-P_{j}f_{0}\|\big).
\]
For any $j\le\mathcal{J}_{\epsilon}$ we then have $\epsilon2^{j(t+d/2)}\to0$
and the concentration inequality from Proposition~\ref{prop:concHatF}
can be applied to $\hat{f}_{j}$ for any $\kappa\in(C^{-1}2^{jd/2}\epsilon,C2^{-jt}\epsilon^{-1})$
for a certain constant $C>0$. We can choose $\kappa=2^{jd/2}\epsilon(\log\epsilon^{-1})$
to obtain
\[
\P_{f_{0},\theta_{0}}(\hat{J}>J_{o})\le6\mathcal{J}_{\epsilon}^{2}e^{-2^{J_{o}d}(\log\epsilon)^{2}}\le6\mathcal{J}_{\epsilon}^{2}\epsilon\to0.
\]

\emph{Step 2:} In order to prove the adaptive contraction rate, we
replace the test $\Psi_{n}$ from the proof of Theorem~\ref{thm:contraction}
by
\[
\tilde{\Psi}:=\1_{\{\|\hat{f}_{\hat{J}}-f_{0}\|\ge2\epsilon(\log\epsilon^{-1})^{2}2^{J_{o}(t+d/2)}\}}
\]
requiring to verify (\ref{eq:test}) for $\tilde{\Psi}$ and 
\begin{equation}
\kappa=(\epsilon\log(1/\epsilon))^{2(s+t)/(2s+2t+d)},\quad\xi\simeq(\log\epsilon^{-1})\big(\epsilon(\log\epsilon^{-1})\big)^{2s/(2s+2t+d)}.\label{eq:kappaXiAdaptive}
\end{equation}
Note that $\epsilon(\log\epsilon^{-1})2^{J_{o}(t+d/2)}\simeq\big(\epsilon(\log\epsilon^{-1})\big)^{2s/(2s+2t+d)}$
by the choice of the oracle $J_{o}$. Thanks to Step 1 we have
\begin{align*}
\E_{f_{0},\theta_{0}}[\tilde{\Psi}] & =\P_{f_{0},\theta_{0}}\big(\|\hat{f}_{\hat{J}}-f_{0}\|\ge2\epsilon(\log\epsilon^{-1})^{2}2^{J_{o}(t+d/2)}\big)\\
 & \le\P_{f_{0},\theta_{0}}\big(\|\hat{f}_{\hat{J}}-f_{0}\|\ge2\epsilon(\log\epsilon^{-1})^{2}2^{J_{o}(t+d/2)},\hat{J}\le J_{o}\big)+6\mathcal{J}_{\epsilon}^{2}\epsilon\\
 & \le\P_{f_{0},\theta_{0}}\big(\|\hat{f}_{J_{o}}-f_{0}\|\ge2\epsilon(\log\epsilon^{-1})^{2}2^{J_{o}(t+d/2)}-\|\hat{f}_{\hat{J}}-\hat{f}_{J_{o}}\|,\hat{J}\le J_{o}\big)+6\mathcal{J}_{\epsilon}^{2}\epsilon.
\end{align*}
By construction of $\hat{J}$ we have $\|\hat{f}_{\hat{J}}-\hat{f}_{J_{o}}\|\le\epsilon(\log\epsilon^{-1})^{2}2^{J_{o}(t+d/2)}$
on the event $\{\hat{J}\le J_{o}\}$. Therefore,
\begin{align*}
\E_{f_{0},\theta_{0}}[\tilde{\Psi}]\le & \P_{f_{0},\theta_{0}}\big(\|\hat{f}_{J_{o}}-f_{0}\|\ge\epsilon(\log\epsilon^{-1})^{2}2^{J_{o}(t+d/2)}\big)+6\mathcal{J}_{\epsilon}^{2}\epsilon\\
\le & \P_{f_{0},\theta_{0}}\big(\|\hat{f}_{J_{o}}-f_{0}\|\ge2CR\epsilon(\log\epsilon^{-1})2^{J_{o}(t+d/2)}\big)+6\mathcal{J}_{\epsilon}^{2}\epsilon\\
\le&3\epsilon+6\mathcal{J}_{\epsilon}^{2}\epsilon\to0
\end{align*}
where the last bound follows from Proposition~\ref{prop:concHatF}
exactly as in Step 1. For any $f\in\mathcal{F}_{n}$ with $\|f-f_{0}\|\ge C_1\epsilon(\log\epsilon^{-1})^{2}2^{J_{0}(t+d/2)}$ for an sufficiently large constant $C_1$
and $\theta\in\Theta$ we obtain on the alternative with an argument
as in (\ref{eq:alternative}) 
\begin{align*}
\E_{f,\theta}[1-\tilde{\Psi}] & =\P_{f,\theta}\big(\|\hat{f}_{\hat{J}}-f_{0}\|\le2\epsilon(\log\epsilon^{-1})^{2}2^{J_{o}(t+d/2)}\big)\\
 & \le\P_{f,\theta}\big(\|\hat{f}_{\hat{J}}-f\|\ge\|f-f_{0}\|-2\epsilon(\log\epsilon^{-1})^{2}2^{J_{o}(t+d/2)}\big)\\
 & \le\P_{f,\theta}\big(\|\hat{f}_{\hat{J}}-f\|\ge C_2(1+\|f\|)\epsilon(\log\epsilon^{-1})^{2}2^{J_{o}(t+d/2)}\big).\\
 & \le3(1+2\mathcal{J}_{\epsilon}^{2})e^{-c2^{J_{o}d}(\log\epsilon)^{2}}
\end{align*}
for some $C_2,c>0$. Since $\mathcal{J}_{\epsilon}\simeq\log\epsilon^{-1}$
and 
\[
(\log\epsilon^{-1})^{2}2^{J_{o}d}\simeq(\log\epsilon^{-1})^{2}(\epsilon\log(1/\epsilon))^{-2d/(2s+2t+d)}\simeq\kappa^{2}\epsilon^{-2},
\]
we indeed have $\E_{f,\theta}[1-\tilde{\Psi}]\le e^{-C'\kappa^{2}\epsilon^{-2}}$
for some constant $C'\ge4$. 

\emph{Step 3:} With the previous preparations we can now prove the
adaptive contraction result. Given $\tilde{\Psi}_{n}$, we have for
any $r>0$ and $\xi$ from (\ref{eq:kappaXiAdaptive})
\begin{align*}
 & \P_{f_{0},\theta_{0}}\Big(\Pi_{n}\big(f\in\mathcal{F}:\|f-f_{0}\|>M\xi|Y,T\big)>r\Big)\\
 & \quad\le\P_{f_{0},\theta_{0}}\Big(\Pi_{n}\big(f\in\mathcal{F}\cap V_{\hat{J}}:\|f-f_{0}\|>M\xi|Y,T\big)(1-\tilde{\Psi})>r,\hat{J}\le J_{o}\Big)+6\mathcal{J}_{\epsilon}^{2}\epsilon\\
 & \quad\le\sum_{j\le J_{0}}\P_{f_{0},\theta_{0}}\Big(\Pi_{n}\big(f\in\mathcal{F}\cap V_{j}:\|f-f_{0}\|>M\xi|Y,T\big)(1-\tilde{\Psi})>r,\hat{J}=j\Big)+6\mathcal{J}_{\epsilon}^{2}\epsilon.
\end{align*}
We can now handle each term in the sum exactly as in the proof of
Theorem~\ref{thm:contraction}. It suffices to note that: First,
$\tilde{\Psi}$ depends only on the $\hat{J}=j$ projection of $Y$
and $j+m$ projection of $\theta$, respectively. Second, if the small
ball probability condition (\ref{eq:smallBall}) is satisfied for
$J_{o}$, as verified in the proof of Theorem~\ref{thm:mildly},
than by monotonicity it is also satisfied for all $j\le J_{0}$. We
thus conclude from (\ref{eq:boundRate})
\[
\P_{f_{0},\theta_{0}}\Big(\Pi_{n}\big(f\in\mathcal{F}:\|f-f_{0}\|>M\xi|Y,T\big)>r\Big)\le\frac{J_{o}}{r}e^{-2\kappa^{2}\epsilon^{-2}}+J_{o}\frac{\epsilon^{2}}{\kappa^2}+6\mathcal{J}_{\epsilon}^{2}\epsilon\to0.\tag*{{\qed}}
\]

\section*{Acknowledgement}

This work is the result of three conferences
which I have attended in 2017. The first two on Bayesian inverse problems
in Cambridge and in Leiden lead to my interest for this problem. On
the third conference in Luminy on the honour of Oleg Lepski's and
Alexandre B. Tsybakov's 60th birthday, I realised that Lepski's method
can be used to construct an empirical Bayes procedure. I want to thank
the organisers of these three conferences. The helpful comments by two anonymous referees are grate-
fully acknowledged.

\bibliographystyle{apalike}
\bibliography{lib}

\end{document}